\documentclass[a4paper,11pt]{article}
\usepackage[english]{babel}
\usepackage{mathrsfs}

\usepackage{amsfonts}
\usepackage[centertags]{amsmath}
\usepackage{amsthm}
\usepackage{enumitem}
\usepackage{latexsym,amsmath,amssymb}
\usepackage{graphicx,color}
\newcommand{\ds}{\displaystyle}

\setenumerate[1]{label=(\roman*), ref=(\roman*)}
\usepackage[all]{xy}

\newtheorem{thm}{Theorem}
\newtheorem{prop}[thm]{Proposition}

\newtheoremstyle{obs}
  {3pt}
  {3pt}
  {}
  {}
  {\bfseries}
  {.}
  {.5em}
  {}
\theoremstyle{obs}
\newtheorem{remark}[thm]{Remark}

\newtheorem{state}[thm]{Statement}

\def\qed{\ifvmode\removelastskip\fi
{\unskip\nobreak\hfil\penalty50\hbox{}\nobreak\hfil \hbox{\vrule
height1.2ex width1.2ex}\parfillskip=0pt \finalhyphendemerits=0
\par \smallskip}}

\title{$k$-symplectic Pontryagin's Maximum Principle for some families of PDEs}
\author{\textsc{M. Barbero\textendash{}Li\~n\'an}\\
Departamento de Matem\'aticas, Universidad Carlos III de Madrid\\ Avenida de la Universidad 30, 28911 Legan\'es, Madrid, Spain,\\ and 
Instituto de Ciencias Matem\'aticas (CSIC-UAM-UC3M-UCM) 
\\[3mm]
\textsc{M. C. Mu\~noz-Lecanda}
\\
 Departamento de Matem\'atica Aplicada IV,  Universitat Polit\`ecnica de Catalunya,\\
 Edificio C-3, Campus Norte UPC,
   C/ Jordi Girona 1. 08034 Barcelona, Spain}
\begin{document}

\maketitle

\begin{abstract}
An optimal control problem associated with the dynamics of the orientation of a bipolar molecule in the plane can be understood by means of tools in differential geometry. For first time in the literature $k$-symplectic formalism is used
to provide the optimal control problems associated to some families of partial differential equations with a geometric formulation. A parallel between the classic formalism of optimal control theory with ordinary differential 
equations and the one with particular families of partial differential equations is established. This description allows us to state and prove 
Pontryagin's Maximum Principle on $k$-symplectic formalism. We also consider the unified Skinner-Rusk
formalism for optimal control
 problems governed by an implicit partial differential equation. 
\end{abstract}

\section{Introduction}

Boscain et al. in~\cite{Marco} study the
controllability of the equation
\begin{equation}
{\rm i} \dfrac{\partial \Psi(t,\theta)}{\partial t}= 
-\dfrac{\partial^2 \Psi(t,\theta)}{\partial \theta^2}+u_1(t) \cos
\theta \, \Psi(t,\theta)+u_2(t) \sin \theta \, \Psi(t,\theta),
\label{eq:ControlIntro}
\end{equation}
which models the rotation motion of a bipolar rigid molecule confined
to a plane with two control electric fields. See references in the paper for more details about the problem origin and its interests. 

The study of controllability in~\cite{Marco} does not approach the problem of existence
and construction of suitable controls for governing the position of
the molecule. The controls are obtained depending on the purpose to be achieved.  For instance,  if the controls are related with the
energy needed to take the molecule to a particular position or to
track a molecule path in the configuration space, then we might be interested in minimizing the energy consumption. In other words we have associated an optimal control problem to the above control equation. Note that control equation is a particular kind of second-order and control-linear partial
differential equation.

In this paper, a geometric approach is considered to deal with optimal control problems for particular families of
control partial differential equations, similar to the above mentioned
example. 

Following the same lines as in our previous paper on Pontryagin's Maximum Principle~\cite{2009BM},
PMP, we extend the geometric method to optimal control problems with some control partial differential equations. 
As in the classical PMP, to succeed in this extension is necessary to extend the control system in a proper way and to optimize a suitable function, the Pontryagin Hamiltonian, instead of the functional.

Before entering into the details, let us provide some historical background on optimal control theory. 
L. S. Pontryagin talked publicly for the first time about the Maximum
Principle in 1958, in the International Congress of Mathematicians
that was held in Edinburgh, Scotland. This Principle was developed by
a research group on automatic control created by Pontryagin in the
fifties. He was engaged in applied mathematics by his friend A.
Andronov and because scientists in the Steklov Mathematical Institute
were asked to carry out applied research, especially in the field of
aircraft dynamics.

At the same time, in the regular seminars on automatic control in the
Institute of Automatics and Telemechanics, A. Feldbaum introduced
Pontryagin and his colleagues to the time-optimization problem. This
allowed them to study how to find the best way of piloting an
aircraft in order to defeat a zenith fire point in the shortest time
as a time-optimization problem.

Since the equations for modelling the aircraft's problem are
nonlinear and the control of the rear end of the aircraft runs over a
bounded subset, it was necessary to reformulate the calculus of
variations known at that time. Taking into account ideas suggested by
E. J. McShane in~\cite{39McShane}, Pontryagin and his collaborators
managed to state and prove the Maximum Principle, which was published
in Russian in 1961 and translated into English~\cite{P62} the
following year. See~\cite{Boltyanski} for more historical remarks.

Although geometric control theory has been studied since the sixties,
geometric optimal control theory started to be developed in the
nineties~\cite{1997Jurdjevic,S98Free}. However, no geometric
description has been made for optimal control problems governed by
partial differential
 equations~\cite{BookOCPPDESConference,BookOCPPDES1,BookOCPPDES2}. In
 this paper we extend for first time the geometric description of
 optimal control problems to those governed by some classes of
 partial differential equations in order to solve optimal
 control problems that can be found in the physical world, as mentioned above.

With this purpose in mind, the natural geometric background to use is 
$k$-symplectic formalism (G\"unther standard polysymplectic), which
is a generalization of the symplectic formalism in classical
mechanics. The $k$-symplectic formalism makes possible to
geometrically interpret some problems such as the vibrating string
within field theory~\cite{Awane} and other problems~\cite{MunozSilvia}. 
Locally speaking, these problems correspond with
Lagrangian and Hamiltonian functions that do not depend on the base
coordinates, usually denoted by $(t^1,\dots,t^k)$. When a dependence
on the base coordinates exists, the $k$-cosymplectic formalism is
necessary. In other words the $k$-cosymplectic formalism is the
generalization of the cosymplectic formalism used in non-autonomous
mechanics to field theories~\cite{deLeon1,deLeon2}.

However, the control equation~\eqref{eq:ControlIntro} under study is of second order. Hence we need to
extend the $k$-symplectic formalism for optimal control problems developed in this paper to  implicit control differential equations. In this framework we transform the equation~\eqref{eq:ControlIntro}  into a first order one in such a way that we obtain an implicit equation. 

The paper is organized as follows: In Section~\ref{Sec:Setting} we
define the setting to describe  optimal control problems governed by
an explicit first-order partial differential equation using the
$k$-symplectic formalism. A parallel between this formulation and the
geometric description of optimal control problems governed by
ordinary differential equations is considered in order to stress the
similarities and differences between both problems.

 One of the key points to prove Pontryagin's Maximum Principle
consists of extending in a suitable way the control system by adding
new coordinates which contain the information related to the cost
function, the so called \emph{extended system}. In this paper, the
$k$-symplectic formalism for the extended optimal control problems
only works under some particular assumptions on the cost function,
which turn out to include the most typical cost functions in the
literature. To prove Pontryagin's Maximum Principle on $k$-symplectic
formalism in Section~\ref{Sec:kPMP} we define the elementary
perturbation vectors on that formalism.

After this first approach to tackle optimal control problems, we are
going to consider in Section~\ref{Sec:UnifiedKSymplectic}
 the unified Skinner-Rusk formalism for $k$-cosymplectic implicit dynamical systems.
Following the lines of~\cite{2007SRNuestro}, we adapt the above formalism to describe a novel unified formalism for optimal control
problems governed by an implicit partial differential
 equation in Section~\ref{Sec:UnifControlImplicitEDPS}. This
generalized unified formalism will allow us to consider interesting
problems associated with higher order control partial differential
equations, in particular the problem that has motivated our study, see Section~\ref{Sec:example}.
In this last section, we consider the  control partial differential
equation that models the orientation of a bipolar molecule in the
plane
 studied in~\cite{Marco}, as described at the beginning,  when a
control-quadratic cost function is considered. Hence, the use of
$k$-symplectic formalism and all the generalizations we have made in
the previous sections to deal with optimal control problems on
partial differential equations are fully justified.

In the sequel, unless otherwise stated, all the manifolds are real, second countable and $C^\infty$. The maps are
assumed to be also $C^\infty$. Sum over all repeated indices is understood.

\section{$k$-symplectic formalism for optimal control problems governed by an explicit first-order partial differential equation} \label{Sec:Setting}

We first recall briefly the essential definitions and notations in the $k$-symplectic formalism. Let $Q$ be a $n$-dimensional manifold and $\tau_Q\colon TQ \rightarrow Q$ be the natural tangent bundle projection. The \textit{$k$-tangent bundle} or the \textit{bundle of $k^1$-velocities} of $Q$, denoted by $T^1_kQ$, is the Whitney sum of $k$
copies of the tangent bundle $TQ$, that is,
\begin{equation}
 T^1_kQ=TQ\oplus \stackrel{k}{\cdots} \oplus TQ.
\label{eq:T1kQ}
\end{equation}
The elements of $T^1_kQ$ are $k$-tuple $(v_{1_q},\dots,v_{k_q})$ of tangent vectors on $Q$ at the same point $q\in Q$.

The canonical projection $\tau^k_Q\colon T^1_kQ \rightarrow Q$ is defined as follows
\begin{equation*}
 \tau^k_Q(v_{1_q},\dots,v_{k_q})=q.
\end{equation*}
If $(V,(q^i))$ is a local chart on $Q$, then it induces a local chart $(T^1_kV,(q^i,v^i_A))$ on $T^1_kQ$, where
$T^1_kV=(\tau^k_Q)^{-1}(V)$.

A \textit{$k$-vector field} on $Q$ is a section $\mathbf{X}\colon Q \rightarrow T^1_kQ$ of the canonical projection
$\tau^k_Q$. Hence a $k$-vector field $\mathbf{X}$ defines a family of $k$ ordinary vector fields $\{X_1,\dots,X_k\}$ on $Q$ through
the canonical projections $\tau^{k;A}_Q\colon T^1_kQ \rightarrow TQ$ onto the A-th component of $T^1_kQ$, that is,
\begin{equation*}
\tau^{k;A}_Q(v_{1_q},\dots,v_{k_q})=v_{A_q}
\end{equation*}
where $A=1,\dots,k$. Note that $X_A= \tau_Q^{k;A}\circ \mathbf{X}$.

An \textit{integral section of $\mathbf{X}$} is a map $\sigma\colon \mathbb{R}^k\rightarrow Q$, $\mathbf{t}=(t^1,\dots,t^k) \rightarrow \sigma(\mathbf{t})$ such that 
\begin{equation*}{\rm T}^1_k\sigma=\left( \frac{\partial \sigma}{\partial t^1}, \dots, \frac{\partial \sigma}{\partial t^k}\right)_{\sigma(\mathbf{t})}=\mathbf{X}\circ \sigma.
\end{equation*}

We introduce now the notion of control system in the $k$-symplectic formalism. Consider a \textit{control set} $U\subset \mathbb{R}^l$. We need the notion of a $k$-vector field $\mathbf{X}$
defined along the projection $\pi_1\colon Q\times U\rightarrow Q$. Such a $k$-vector field is defined by making the following
diagram commutative:
\begin{equation*}
 \xymatrix{ && T^1_k Q  \ar[d]^{\txt{\small{$\tau^k_Q$}}} \\ J\subseteq \mathbb{R}^k \ar[r]^{\txt{\small{$\phi$}}}\ar@/^/[urr]^{\txt{\small{$({\rm T}^1_k)(\pi_1\circ \phi) \quad$}}}&  Q\times U   \ar[ur]^{\txt{\small{$\mathbf{X}$}}}
 \ar[r]^{\txt{\small{$\pi_1$}}} & Q}
\end{equation*}
where $J$ is a subset of $\mathbb{R}^k$.

An \textit{integral section of a $k$-vector field $\mathbf{X}$ defined along the projection $\pi_1\colon Q\times U\rightarrow Q$}
 is a map $\phi=(\sigma,u) \colon J\subseteq
\mathbb{R}^k \rightarrow Q\times U$ such that
\begin{equation*}
({\rm T}^1_k)(\pi_1\circ \phi)=\mathbf{X}\circ \phi,
\end{equation*}
or in other terms,
\begin{equation*}
 {\rm T}_{(t^1,\dots,t^k)} (\pi_1\circ \phi) \frac{\partial}{\partial t^A}= X_A(\phi(t^1,\dots,t^k))=X_A(\sigma(t^1,\dots,t^k),u(t^1,\dots,t^k)),
\end{equation*}
equivalently $ (\pi_1\circ \phi)_*  \frac{\partial}{\partial t^A}=X_A \circ \phi$
for $A=1,\dots, k$.

Let $F\colon Q\times U\rightarrow \mathbb{R}$ be a regular enough map. Such a function, which is usually called the \textit{cost function} in the
literature, allows us to define the functional
\begin{equation}
{\mathcal F}[ \phi ] =\int_{{\rm Dom}\; \phi} (F\circ \phi) {\rm d}S,
 \label{eq:Functional}
\end{equation}
where ${\rm d}S={\rm d}t^1\wedge \dots \wedge {\rm d}t^k$, i.e. the usual volume form in $\mathbb{R}^k$.

From now on, we assume that ${\rm Dom}\; \phi=I_1\times \dots \times I_k=[t_0^1,t_f^1] \times \dots \times [t_0^k,t_f^k]=\colon \mathbf{I}$.

Before stating the optimal control problem on $k$-symplectic formalism, we remind here the classical optimal control problem 
with cost function $G\colon  Q\times U \rightarrow \mathbb{R}$.

\begin{state}[\textbf{Optimal control problem, OCP}] Given $(Q,U,X,G,I)$. Find a curve $(\gamma,u)\colon I\subset \mathbb{R} \rightarrow Q\times U$ joining the points $x_0$ and
$x_f$ in $Q$ such that
\begin{enumerate}
 \item it is an integral curve of the vector field $X$ defined along the projection $\pi_1\colon Q\times U \rightarrow Q$, i.e.
$\dot{\gamma}(t)=X(\gamma(t),u(t))$;
\item it minimizes the functional $\int_I G(\tilde{\gamma}(t),\tilde{u}(t)){\rm d}t$ among all the integral curves $(\tilde{\gamma}, \tilde{u})$ of $X$ on $Q\times U$ joining
$x_0$ and $x_f$.
\end{enumerate}  \label{State:OCP}
\end{state}

\begin{state}[\textbf{$k$-symplectic optimal control problem, $k$-OCP}] Given $(Q,U,\mathbf{X},F,\mathbf{I})$. Find a map $\phi=(\sigma,u)\colon\mathbf{I}= I_1\times \dots \times I_k \subset \mathbb{R}^k \rightarrow Q\times U$
passing  through the points $q_0$ and
$q_f$ in $Q$ such that
\begin{enumerate}
 \item it is an integral section of the $k$-vector field $\mathbf{X}=(X_1,\dots,X_k)$ defined along the
projection $\pi_1\colon Q\times U \rightarrow Q$, i.e.
\begin{equation} {\rm T}^1_k(\pi_1\circ \phi)=\mathbf{X} \circ \phi,\quad {\rm i. e.} \quad \frac{\partial \sigma^i}{\partial t^A}(\mathbf{t})=X^i_A(\phi(\mathbf{t}))=
X^i_A(\sigma(\mathbf{t}),u(\mathbf{t})), \label{Eq:KSymplOCPeq}
 \end{equation}
where $\mathbf{t}=(t^1,\dots,t^k)\in I_1\times \dots \times I_k$;
\item it minimizes the functional $\int_{I_1\times \dots \times I_k} F(\tilde{\phi}(\mathbf{t})){\rm d}^k\mathbf{t}$
among all the integral sections $\tilde{\phi}$
of $\mathbf{X}$ on $Q\times U$ passing through $q_0$ and $q_f$, where ${\rm d}^k\mathbf{t}={\rm d}t^1\wedge \dots \wedge {\rm d}t^k$ .
\end{enumerate} \label{State:kOCP}
\end{state}

Let us compare the frameworks in the traditional optimal control problems and in the k-symplectic optimal control problems.

\begin{minipage}[t]{0.45\textwidth}
\begin{center}
Classical OCP for ODE
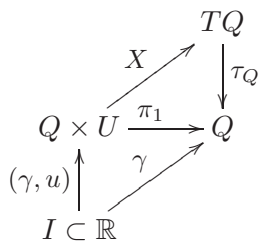
\begin{equation*}\xymatrix{ & TQ \ar[d]^{\txt{\small{$\tau_Q$}}} \\
Q\times U \ar[r]^{\txt{\small{$\pi_1$}}} \ar[ru]^{\txt{\small{$X$}}} & Q  \\
I\subset \mathbb{R} \ar[u]^{\txt{\small{$(\gamma,u)$}}} \ar[ru]^{\txt{\small{$\gamma$}}}  &
}\end{equation*}
\end{center}
\end{minipage}
\hfill
\begin{minipage}[t]{0.45\textwidth}
\begin{center}
$k$-symplectic OCP
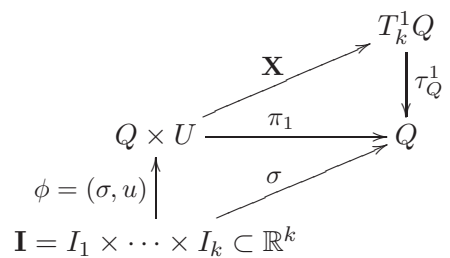
\begin{equation*}\xymatrix{ &T^1_kQ \ar[d]^{\txt{\small{$\tau^1_Q$}}} \\
 Q\times U \ar[r]^{\txt{\small{$\pi_1$}}}
\ar[ru]^{\txt{\small{$\mathbf{X}$}}} & Q  \\
\mathbf{I}=I_1\times \dots \times I_k\subset \mathbb{R}^k \ar[u]^{\txt{\small{$\phi=(\sigma,u)$}}} \ar[ru]^{\txt{\small{$\sigma$}}} &
}\end{equation*}
\end{center} \end{minipage}

Observe that the classical OCP has associated a problem of explicit ordinary differential equations, whereas the equations in the $k$-symplectic optimal control problem are explicit partial differential equations.

Since late fifties the most efficient tool to solve optimal control problem is Pontryagin's Maximum Principle, which provides us with necessary conditions for optimality~\cite{P62}. One of the key points to prove that Principle for classical optimal control
theory consists of extending the control system in a suitable way. To be more precice, $Q$ is extended to the manifold
$\widehat{Q}=\mathbb{R}\times Q$ with local coordinates $\widehat{x}=(x^0,x^i)$ and the corresponding extended vector field is given by
\begin{equation*}
 \widehat{X}(\widehat{x},u)=G(x,u)\frac{\partial}{\partial x^0}_{(\widehat{x},u)}+X(x,u).
\end{equation*}
Note that the system of ordinary differential equations which determines the integral curves of $\widehat{X}$ 
can be decoupled in the following sense: we first integrate $\dfrac{{\rm d} x^i}{{\rm d}t}=X^i(x,u)$ and then we have
\begin{equation*}
 x^0(t)=\int_{t_0}^t G(\gamma(s),u(s)){\rm d}s,
\end{equation*}
for any $t\in I=[t_0,t_1]$. 

Unfortunately, in order to extend coherently the optimal control problem on $k$-symplectic formalism we need some extra assumptions on the cost function. 

\textbf{Assumption 1.} The Lie derivative of the cost function with respect to each $X_A$ is zero, that is, ${\rm L}_{X_A} F=0$.

\textbf{Assumption 2.} The control functions $u\colon \mathbf{I} \rightarrow U$ are locally constants. 

We can justify these assumptions as follows. In a first try to extend the control system we will add $k$ new 
variables $(q^{0_1},\dots,q^{0_k})$ such that for every $A=1,\dots,k$
\begin{eqnarray}
 \dfrac{\partial q^{0_A}}{\partial t^A}&=& F, \label{eq:Q0AtA} \\
\dfrac{\partial q^{0_A}}{\partial t^B}&=& 0,\quad \mbox{ for } B\neq A.  \label{eq:Q0AtB}
\end{eqnarray}
Once we have an integral section of $\mathbf{X}$, we integrate~\eqref{eq:Q0AtA} and obtain
\begin{equation*}
 q^{0_A}[t^1,\dots,\widehat{t^A},\dots,t^k](t^A)=\int^{t^A}_{t_0^A} F(q^1,\dots,q^n,u^1, \dots, u^l)(t^1,\dots,s,\dots,t^k)\, {\rm d}s.
\end{equation*}
The equation~\eqref{eq:Q0AtB} is satisfied by $q^{0_A}$ for $B\neq A$ if 
\begin{equation*}
 \dfrac{\partial q^{0_A}}{\partial t^B}=\int^{t^A}_{t^A_0} 
\left(\dfrac{\partial F}{\partial q^i}\, \dfrac{\partial q^i}{\partial t^A}
+ \dfrac{\partial F}{\partial u^a}\, \dfrac{\partial u^a}{\partial t^A}
  \right)(t^1,\dots,s,\dots,t^k) {\rm d}s=0.
\end{equation*}
Note that if both assumptions are satisfied the equations will be immediately satisfied. Thus these assumptions are necessary to 
guarantee the compatibility of the system of partial differential equations when we extend the control system.

Moreover, having in mind~\cite{Grasse} Assumption 2 is reasonable when dealing with control systems. These two assumptions include the most typical cost functions considered in optimal control problems such as
control-quadratic, constant function 1 (that is, time optimal), etc. Hence these assumptions do not impose great restrictions
according to the literature.

Under the above assumptions, let us consider now the extended $k$-symplectic optimal control problem. In order to preserve
the same philosophy as in classical control theory, we will have to add $k$ new coordinates $(q^{0_B})_{\{B=1,\dots,k\}}$. Then the extended manifold in $k$-symplectic formalism is given by $\widehat{Q}=\mathbb{R}^k\times Q$.
If $\mathbf{X}=(X_1,\dots,X_k)$ is the $k$-vector field on $Q$, then the \textit{extended $k$-vector field
$\widehat{\mathbf{X}}$ on $\widehat{Q}$} is given by $(\widehat{X}_1,\dots,\widehat{X}_k)$ where
\begin{equation}
 \widehat{X}_A=F \delta^B_A\, \ds{\frac{\partial}{\partial q^{0_B}}}+X_A=F\, \ds{\frac{\partial}{\partial q^{0_A}}}+X_A,\quad {\rm for \; every} \quad A=1,\dots,k, \label{eq:hatXA}
\end{equation}
where $\delta^B_A$ is the Kronecker's delta and $F$ is the cost function.

As in classical optimal control theory, we can first integrate $\dfrac{\partial q^i}{\partial t^A}=X^i_A(q,u)$ and obtain $\phi=(\sigma,u)\colon \mathbf{I} \rightarrow Q\times U$. Then,  
\begin{equation*}
q^{0_A}(\mathbf{t})= \int^{t^A}_{t^A_0} F(\sigma(\mathbf{t}),u(\mathbf{t})){\rm d}t^A, \quad \mbox{for every } A=1,\dots, k.
\end{equation*}
Because of~\eqref{eq:Q0AtB} $q^{0_A}$ is constant when we fix $t^A$. On the other hand, due to Assumption 1 and~\ref{eq:Q0AtA} $q^{0_A}$ is constant along integral curves of $X_A$ for every $A=1,\dots, k$.

\begin{state}[\textbf{Extended $k$-symplectic optimal control problem}] Given $(\widehat{Q},U,\widehat{\mathbf{X}},F,\mathbf{I})$. Find a map $\widehat{\phi}=(\widehat{\sigma},u)
\colon \mathbf{I}\subset \mathbb{R}^k \rightarrow \widehat{Q}\times U$
passing through the points $(\mathbf{0},q_0)$ in $\widehat{Q}$ and
$q_f$ in $Q$ such that
\begin{enumerate}
 \item it is an integral section of the $k$-vector field $\widehat{\mathbf{X}}=(\widehat{X}_1,\dots,\widehat{X}_k)$
defined along the
projection $\widehat{\pi}_1\colon \widehat{Q}\times U \rightarrow \widehat{Q}$, i.e. locally
\begin{equation*} {\rm T}^1_k(\widehat{\pi}_1 \circ \widehat{\phi})=  \widehat{\mathbf{X}} \circ \widehat{\phi}, \quad {\rm i.e.} \quad \frac{\partial \sigma^{0_B}}{\partial t^A}(\mathbf{t})=F(\phi(\mathbf{t}))\delta^B_A,
\quad \frac{\partial \sigma^i}{\partial t^A}(\mathbf{t})=X^i_A(\phi(\mathbf{t})),
 \end{equation*}
where $\mathbf{t}=(t^1,\dots,t^k)\in I_1\times \dots \times I_k$, $\delta^B_A$ is the Kronecker's delta, for every $A,B=1,\dots,k$;
$i=1,\dots,n$;
\item it minimizes each functional \begin{equation}\label{eq:FA}
 {\mathcal F}_A[\phi] (\mathbf{t})=\int^{t^A}_{t^A_0} F(\phi(t^1,\dots, \stackrel{{\rm Ath}}{s}, \dots, t^k)) \, {\rm d}s,
\end{equation} for $A=1,\dots, k$,
among all the integrals sections $\widehat{\phi}$
of $\widehat{\mathbf{X}}$ on $\widehat{Q}\times U$ passing through $q_0$ and $q_f$ such that $\phi=\pi_{Q\times U}\circ \widehat{\phi}$ for $\pi_{Q\times U}\colon \widehat{Q}\times U \rightarrow Q\times U$.
\end{enumerate}\label{State:HatkOCP}
\end{state}

\remark \label{RemNoEquivalence} Note that if the projection of an integral section $\widehat{\phi}\colon I_1\times\dots \times I_k \rightarrow \widehat{Q}\times U$ of
$\widehat{\mathbf{X}}$ to $\phi$ on $Q\times U$ minimizes each functional in~\eqref{eq:FA}, then the projection of the integral
section $\widehat{\phi}$ to $\phi$ on $Q\times U$ minimizes the functional
\begin{equation}
{\mathcal F}[\phi] =\int_{I_1\times\dots \times I_k} (F\circ \phi) {\rm d}^k\mathbf{t} \label{eq:F}
\end{equation}
since the order of integration does not matter.

Hence, in contrast with classical optimal control theory, in $k$-symplectic formalism the extended optimal control problem
and the optimal control problem are not equivalent. However, solutions to the extended problem in Statement~\ref{State:HatkOCP} are also solutions to the original  $k$-symplectic optimal control problem in Statement~\ref{State:kOCP}.  As we will see
later on,  the adapted version of Pontryagin's Maximum Principle in
$k$-symplectic formalism provides us with necessary conditions for optimality of the $k$ functionals in~\eqref{eq:FA} for those cost functions satisfying assumptions 1 and 2.

As mentioned above the trajectories that minimize~\eqref{eq:FA} also minimize~\eqref{eq:F}, but not necessarily in the other way around. Remember that to minimize a multiple integral does not imply that every simple integral involved is minimized. Thus, the necessary conditions for optimality described in Section~\ref{Sec:kPMP} in the $k$-symplectic version of Pontryagin's Maximum Principle are more restrictive than the traditional necessary conditions for optimality in~\cite{2009BM,P62}.

The elements of extended optimal control problems in classical formalism and $k$-symplectic formalism are
summarized in the following diagrams:

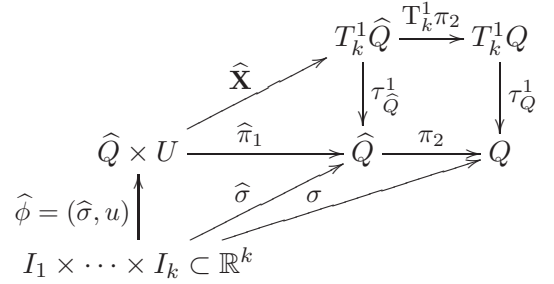
\begin{minipage}[t]{0.45\textwidth}
\begin{center}
Classical extended OCP for ODE
\begin{equation*}\xymatrix{   & T\widehat{Q} \ar[d]^{\txt{\small{$\tau_{\widehat{Q}}$}}} \ar[r]^{\txt{\small{${\rm T}\pi_2$}}}
& TQ  \ar[d]^{\txt{\small{$\tau_{Q}$}}}
\\
\widehat{Q}\times U \ar[r]^{\txt{\small{$\widehat{\pi}_1$}}} \ar[ru]^{\txt{\small{$\widehat{X}$}}} & \widehat{Q}
\ar[r]^{\txt{\small{$\pi_2$}}} & Q  \\
I\subset \mathbb{R} \ar[u]^{\txt{\small{$(\widehat{\gamma},u)$}}} \ar[ru]^{\txt{\small{$\widehat{\gamma}$}}}
\ar[rru]^{\txt{\small{$\gamma$}}} & &
}\end{equation*}
\end{center}
\end{minipage}
\hfill{}
\begin{minipage}[t]{0.45\textwidth}
\begin{center}
$k$-symplectic extended OCP
\begin{equation*}\xymatrix{
&T^1_k\widehat{Q} \ar[d]^{\txt{\small{$\tau^1_{\widehat{Q}}$}}} \ar[r]^{\txt{\small{${\rm T}^1_k\pi_2$}}}
& T^1_k Q  \ar[d]^{\txt{\small{$\tau^1_Q$}}} \\  \widehat{Q}\times U \ar[r]^{\txt{\small{$\widehat{\pi}_1$}}}
\ar[ru]^{\txt{\small{$\widehat{\mathbf{X}}$}}} & \widehat{Q}  \ar[r]^{\txt{\small{$\pi_2$}}} & Q \\
I_1\times \dots \times I_k\subset \mathbb{R}^k \ar[u]^{\txt{\small{$\widehat{\phi}=(\widehat{\sigma},u)$}}}
\ar[ru]^{\txt{\small{$\widehat{\sigma}$}}}  \ar[rru]^{\txt{\small{$\sigma$}}} & &
}\end{equation*}
\end{center}
\end{minipage}

In order to state Pontryagin's Maximum Principle we need a Hamiltonian problem associated with each of the extended optimal control problems.
It is important to remark here that this Hamiltonian problem is not equivalent to the optimal control problems in the classical formalism neither in the $k$-symplectic formalism. 

%

The Hamiltonian for the extended optimal control problem in classical theory is given by
$H\colon T^*\widehat{Q}\times U \rightarrow \mathbb{R}$,
\begin{equation*}H(\widehat{p},u)=\langle \widehat{p}, \widehat{X}(\widehat{x},u)\rangle=p_0 G(x,u)+p_iX^i(x,u).
\end{equation*}
For each control $u$, the Hamiltonian vector field $\widehat{X}_H^{\{u\}}=\widehat{X}_H(\cdot,u)$ satisfies the following
Hamilton's equation
\begin{equation*}
 {\rm i}_{\widehat{X}_H^{\{u\}}} \omega ={\rm d}H^{\{u\}},
\end{equation*}
where $\omega$ is the canonical symplectic structure on $T^*\widehat{Q}$. Locally $\omega={\rm d}x^0\wedge {\rm d}p_0 + {\rm d}x^i \wedge  {\rm d}p_i$ in natural local coordinates $(x^0,x^i,p_0,p_i)$ in $T^*\widehat{Q}$.

For the extended $k$-symplectic optimal control problem we consider $k$ Hamiltonian functions
$H_A\colon (T^1_k)^*\widehat{Q}\times U
\rightarrow \mathbb{R}$ defined as follows
\begin{eqnarray}H_A(\widehat{\mathbf{p}},u)&=&\langle \widehat{p}^A, \widehat{X}_A(\widehat{q},u)\rangle=
\sum_{B=1}^k p^A_{0_B} F(q,u) \delta^B_A +\sum_{j=1}^n p^A_j X_A^j(q,u) \nonumber \\
&=& p^A_{0_A}F(q,u)+  \sum_{j=1}^n p^A_j X_A^j(q,u), \label{eq:HA}
\end{eqnarray}
in natural local coordinates $(q^{0_1},\dots, q^{0_k},q^1,\dots,q^n,(p^A_{0_1},\dots, p^A_{0_k},p^A_1,\dots,p^A_n)_{A=1,\dots,k})$
for $(T^1_k)^*\widehat{Q}$.

For each control $u$, the Hamiltonian $k$-vector field $\widehat{\mathbf{X}}^{*^{\{u\}}}=\left(\widehat{X}_{1}^{*^{\{u\}}}, \dots, \widehat{X}_{k}^{*^{\{u\}}}\right)$ must satisfy the following family of equations
\begin{equation}
 {\rm i}_{\widehat{X}_{A}^{*^{\{u\}}}}\omega^A={\rm d}H_A^{\{u\}} \quad \mbox{for every } A=1,\dots, k.\label{eq:iXAHA}
\end{equation}
The canonical $k$-symplectic structure on $(T^1_k)^*\widehat{Q}$ is given by $(\omega_1,\dots, \omega_k)$ where
$\omega_A=(\pi^A)^*\omega$, $\pi^A\colon (T^1_k)^*\widehat{Q} \rightarrow T^*\widehat{Q}$ is the projection onto the Ath-copy and $\omega$ is the
canonical symplectic structure on $T^*\widehat{Q}$. Locally $\omega_A={\rm d} q^{0_j}\wedge {\rm d} p^A_{0_j}+{\rm d} q^i \wedge {\rm d} p^A_i$.

If for each control the Hamiltonian $k$-vector field $\widehat{\mathbf{X}}^{*^{\{u\}}}$ is solution to~\eqref{eq:iXAHA}, then
it is solution to the following Hamilton-De Donder-Weyl equations
\begin{equation}
 \sum_{A=1}^k{\rm i}_{\widehat{X}_{A}^{*^{\{u\}}}}\omega^A=\sum_{A=1}^k{\rm d}H_A^{\{u\}}={\rm d}\left(\sum_{A=1}^kH_A^{\{u\}}\right)={\rm d}\mathbf{H}, \label{eq:SumiXAHA}
\end{equation}
associated with the Hamiltonian $\mathbf{H} \colon (T^1_k)^*\widehat{Q}\times U
\rightarrow \mathbb{R}$ given by
 \begin{equation*}\mathbf{H}(\widehat{\mathbf{p}},u)=\sum_{A=1}^kH_A(\widehat{\mathbf{p}},u)=
\sum_{A=1}^k \langle \widehat{p}^A, \widehat{X}_A(\widehat{q},u)\rangle,
\end{equation*}
where $\widehat{\mathbf{p}}\in (T^1_k)^*_{\widehat{q}} \widehat{Q}$.
By the superposition principle all the solutions of~\eqref{eq:iXAHA} are solutions to~\eqref{eq:SumiXAHA} because both systems are linear in the momenta. However,~\eqref{eq:SumiXAHA} has more solutions apart from the ones coming from~\eqref{eq:iXAHA}.
In fact, for every $A\in \{1,\dots, k\}$ the Ath vector field $\widehat{X}_A^*$ of the Hamiltonian $k$-vector field $\widehat{\mathbf{X}}^*=\left( \widehat{X}_1^*,\dots, \widehat{X}_k^* \right)$
is locally expressed as follows
\begin{equation*}
 \widehat{X}_A^*=(Y_A)^{0_B} \, \ds{\frac{\partial}{\partial q^{0_B}}+ (Y_A)^i\, \frac{\partial}{\partial q^i}
+(Y_A)^C_{0_B} \, \frac{\partial}{\partial p^C_{0_B}}+(Y_A)^C_j \, \frac{\partial}{\partial p^C_j}}.
\end{equation*}
From~\eqref{eq:iXAHA} we obtain
\begin{equation}
 \begin{array}{lcllcl}
  (Y_A)^{0_B}&=&X^{0_B}_A=F \delta^B_A, & (Y_A)^A_{0_B}&=&0,\\
(Y_A)^i&=&X^i_A, & (Y_A)^A_i&=&  \ds{-p^A_{0_A} \, \frac{\partial F}{\partial q^i}-p^A_{j} \, \frac{\partial X^{j}_A}{\partial q^i}}, 
 \end{array} \label{eq:LocaliXAHA}
\end{equation}
for every $A=1,\dots, k$. Note that the Hamiltonian $k$-vector field  $\widehat{\mathbf{X}}^*=\left(  \widehat{X}_1^*,\dots,  \widehat{X}_k^* \right)$  is
not completely determined because the following functions
\begin{equation}
 (Y_A)^C_{0_B}, \quad (Y_A)^C_j \label{eq:wishToBe0}
\end{equation}
remain undetermined for $C\neq A$ and for every $A=1,\dots, k$. 

On the other hand, from~\eqref{eq:SumiXAHA}  the Hamiltonian $k$-vector field $\widehat{\mathbf{X}}^*=\left( \widehat{X}_1^*,\dots, \widehat{X}_k^* \right)$ must satisfy 
\begin{align}
  (Y_A)^{0_B}&=X^{0_B}_A=F \delta^B_A, \quad  (Y_A)^A_{0_B}=0,\quad (Y_A)^i=X^i_A,\nonumber\\
  \sum_{A=1}^k (Y_A)^A_i&=  \sum_{A=1}^k\left(\ds{-p^A_{0_A} \, \frac{\partial F}{\partial q^i}-p^A_{j} \, \frac{\partial X^{j}_A}{\partial q^i}}\right).  \label{eq:LocalSumiXAHA}
\end{align}

By comparing ~\eqref{eq:LocaliXAHA} and~\eqref{eq:LocalSumiXAHA} it is clear that all the solutions to~\eqref{eq:iXAHA} are also solution to~\eqref{eq:SumiXAHA}, but not in the other way around. Neither the Hamiltonian $k$-vector field $\widehat{\mathbf{X}}^{*^{\{u\}}}$ solution to~\eqref{eq:iXAHA}, nor the Hamiltonian $k$-vector field $\widehat{\mathbf{X}}^{*^{\{u\}}}$ solution to the Hamilton-De Donder-Weyl equations are fully determined. For the first one, the functions in~\eqref{eq:wishToBe0} remain undetermined. For the second one, the functions in~\eqref{eq:wishToBe0} remain undetermined and maybe some of the $(Y_A)^A_i$ involved in~\eqref{eq:LocalSumiXAHA}.

However, we can reduce in an intrinsic way the number of functions that remain undetermined in the above mentioned Hamiltonian $k$-vector fields $\widehat{\mathbf{X}}^{*^{\{u\}}}$ in such a way that the  Hamiltonian $k$-vector field $\widehat{\mathbf{X}}^{*^{\{u\}}}$ solution to~\eqref{eq:iXAHA} is fully determined. Note that 
$(T^1_k)^*\widehat{Q}=(T^1_k)^*(\mathbb{R}^k\times Q)\simeq (T^1_k)^*\mathbb{R}^k \times (T^1_k)^* Q$, which has 
two natural projections ${\rm pr}_1$ and ${\rm pr}_2$ from $(T^1_k)^*\widehat{Q} $ to $(T^1_k)^*\mathbb{R}^k$ and $(T^1_k)^* Q$, 
respectively. Consider now the canonical projections $\pi^{k;C}_Q \colon (T^1_k)^*Q \rightarrow T^*Q$ and 
$\pi^{k;C}_{\mathbb{R}^k} \colon (T^1_k)^*\mathbb{R}^k \rightarrow T^*\mathbb{R}^k$
to the $C$th component of $(T^1_k)^*Q$ and $(T^1_k)^*\mathbb{R}^k$, respectively.

The conditions
\begin{eqnarray}
{\rm T}\left(\pi_{\mathbb{R}^k}^{k;C} \circ {\rm pr}_1 \right) \left( \widehat{X}^*_A\right)&=&0, \label{eq:0Condition1}\\
{\rm T}\left(\pi_Q^{k;C} \circ {\rm pr}_2 \right) \left( \widehat{X}^*_A\right)&=&0,
\label{eq:0Condition2} \end{eqnarray} for every $C\neq A$ imply locally that $(Y_A)^C_{0B}=0$  and $(Y_A)^C_j=0$ for $C\neq A$ and for every $C\neq A$.

Under conditions~\eqref{eq:0Condition1},~\eqref{eq:0Condition2}, given an initial condition $\widehat{\mathbf{\beta}}_0$ in $(T^1_k)^*\widehat{Q}$ there exists a unique integral section $\widehat{\mathbf{\beta}}\colon I_1 \times \dots \times I_k \rightarrow (T^1_k)^*\widehat{Q}$ of the Hamiltonian $k$-vector field solution to~\eqref{eq:iXAHA}. It is clear from the local equations~\eqref{eq:LocaliXAHA} that once  $A$ is fixed,  $p^B$ does not appear in the set of equations in~\eqref{eq:iXAHA} associated with $A$ and only $p^A$'s appear.
%
%
%

\subsection{Elementary perturbation vectors and Pontryagin's Maximum Principle on $k$-symplectic formalism} \label{Sec:kPMP}

Now let us introduce the notion of elementary perturbation in $k$-symplectic formalism that allows us to define later
the $k$-symplectic tangent perturbation cones. These elements are essential to prove the $k$-symplectic Pontryagin's
Maximum Principle, Theorem~\ref{thm:kPMP}. 

First fix a surface $(\widehat{\sigma},u)\colon I_1\times \dots \times I_k \rightarrow
\widehat{Q}\times U$. Let $\pi_A$ be  a 3-tuple $\{r_A,l_A,u_A\}$ where $r_A$, $l_A\in\mathbb{R}$ and $u_A\in U\subseteq \mathbb{R}^l$. The \textbf{Ath-elementary perturbation of  the control $u$} is defined as follows
\begin{equation}
 u[\pi_A^s](t^1,\dots,t^k)=\left\{ \begin{array}{ll} u_A, & t^A \in [r_A-l_As,r_A], \\
u(t^1,\dots,t^k), & \rm{elsewhere}.  \end{array} \right.
\end{equation}

Associated to this control $ u[\pi_A^s]$, the mapping $\widehat{\sigma}[\pi_A^s]\colon I_1\times \dots \times I_k \rightarrow
\widehat{Q}$ is the integral section of the $k$-vector field $\widehat{\mathbf{X}}^{\{ u[\pi_A^s]\}}$ with initial
condition $(t_0^1,\dots,t_0^k,\widehat{\sigma}(t_0^1,\dots,t_0^k))$.

Given $\epsilon>0$, define the map
\begin{equation*}
\begin{array}{rcl}
 \varphi_{\pi_A}\colon I_1\times \dots \times I_k \times [0,\epsilon] &\rightarrow &\widehat{Q}\\
(\mathbf{t},s) &\longmapsto & \varphi_{\pi_A}(\mathbf{t},s)=\widehat{\sigma}[\pi_A^s](\mathbf{t}).
\end{array}
\end{equation*}
For every $\mathbf{t}\in I_1\times \dots \times I_k $, $\varphi_{\pi_A}^{\mathbf{t}}\colon [0,\epsilon] \rightarrow \widehat{Q}$
is given by $\varphi_{\pi_A}^{\mathbf{t}}(s)=\varphi_{\pi_A}(\mathbf{t},s)$. The curve $\varphi_{\pi_A}^{\mathbf{t}}$
depends continuously on $s$ and on $\pi_A=\{r_A,l_A,u_A\}$.

From $\widehat{\sigma}[\pi_A^s]$ we can define a curve as follows
\begin{equation*}
\begin{array}{rcl}
 \widehat{\sigma}[\pi_A^s](t^1,\dots, \hat{t^A}, \dots, t^k)\colon I_A &\longrightarrow & \widehat{Q}\\
t^A & \longmapsto & \widehat{\sigma}[\pi_A^s](t^1,\dots, \hat{t^A}, \dots, t^k)(t^A)=
\widehat{\sigma}[\pi_A^s](t^1,\dots, t^A, \dots, t^k).
\end{array}
\end{equation*}
This curve is an integral curve of $\widehat{X}_A^{\{u[\pi_A^s]\}}$ with initial condition
$(t_0^A, \widehat{\sigma}(t_0^1,\dots,t_0^k))$.

\begin{prop} Let $r_A\in I_A$. If $u[\pi_A^s]$ is an elementary perturbation of $u$ specified by the data
$\pi_A=\{r_A,l_A,u_A\}$, then the curve $\varphi_{\pi_A}^{\mathbf{t}}$ is differentiable at $s=0$ and its tangent vector is
\begin{equation}[\widehat{X}_A(\widehat{\sigma}(t^1,\dots, r_A,\dots, t^k),u_A)-\widehat{X}_A(\widehat{\sigma}(t^1,\dots, r_A,\dots, t^k),
u(t^1,\dots, r_A,\dots, t^k))]l_A =\colon \widehat{v}[\pi_A] \label{Eq:PertVectorA} \end{equation}for fixed $(t^1,\dots, r_A,\dots, t^k)\in I_1\times \dots \times I_k$.
\end{prop}

\begin{proof}
In local coordinates $(q^{0_1},\dots, q^{0_k},q^1,\dots, q^n)$ for $\widehat{Q}$, note that
 \begin{multline*}
   q^i\circ \widehat{\sigma}[\pi_A^s](t^1,\dots, \hat{t^A}, \dots, t^k)(r_A)- q^i\circ \widehat{\sigma}[\pi_A^s](t^1,\dots, \hat{t^A}, \dots, t^k)(t_0^A)
\\=\int_{t_0^A}^{r_A} X_A^i( \widehat{\sigma}[\pi_A^s](t^1,\dots, \hat{t^A}, \dots, t^k)(t),u[\pi_A^s](t^1,\dots,t,\dots t^k))\, {\rm d}t
 \end{multline*}
for every $i\in \{0_1,\dots , 0_k, 1,\dots n\}$.

To compute the derivative of  $\varphi_{\pi_A}^{\mathbf{t}}$  at $s=0$ with $\mathbf{t}=(t^1,\dots, r_A,\dots, t^k)$ we use the definition of the derivative:
\begin{eqnarray*}
\left.\ds{\frac{\rm d}{{\rm d}s}}\right|_{s=0}(q^i\circ\varphi_{\pi_A}^{\mathbf{t}}) (s)&=&\lim_{s\rightarrow 0}
\ds{\frac{(q^i\circ\varphi_{\pi_A}^{\mathbf{t}}) (s)-(q^i\circ\varphi_{\pi_A}^{\mathbf{t}}) (0)}{s}}\\&=&
\lim_{s\rightarrow 0}
\ds{\frac{q^i\circ \widehat{\sigma}[\pi_A^s](t^1,\dots, r_A, \dots, t^k)-q^i\circ\widehat{\sigma}(t^1,\dots, r_A, \dots, t^k)}{s}}\\
 &=& \lim_{s\rightarrow 0} \left(\ds{\frac{\int^{r_A}_{t_0^A} X_A^i( \widehat{\sigma}[\pi_A^s](t^1,\dots, \hat{t^A}, \dots, t^k)(t),u[\pi_A^s](t^1,\dots,t,\dots t^k))
\, {\rm d}t}{s}}\right.
\end{eqnarray*}
\begin{eqnarray*} &&- \left.\ds{\frac{\int^{r_A}_{t_0^A} X_A^i( \widehat{\sigma}(t^1,\dots, \hat{t^A}, \dots, t^k)(t),u(t^1,\dots,t,\dots t^k))
\, {\rm d}t}{s}}\right)\\
&=&  \lim_{s\rightarrow 0}\left( \ds{\frac{\int^{r_A}_{r_A-l_As} X_A^i( \widehat{\sigma}[\pi_A^s](t^1,\dots, \hat{t^A}, \dots, t^k)(t),u_A)
\, {\rm d}t}{s}}\right. \\ &&-\left.\ds{\frac{\int^{r_A}_{r_A-l_As} X_A^i( \widehat{\sigma}(t^1,\dots, \hat{t^A}, \dots, t^k)(t),u(t^1,\dots,t,\dots t^k))
\, {\rm d}t}{s}}\right)={\mathcal C}.
\end{eqnarray*}

Let us use now the following equation
\begin{equation}
\int^t_{t-s} X(\gamma(h),u(h)){\rm d}h=s X(\gamma(t),u(t))+o(s), \label{eq:LebesgueTime}
\end{equation}
in the above formula having in mind that $o(s)$ tends to 0 when $s$ tends to 0. Then,
\begin{eqnarray*}
{\mathcal C}&=& \lim_{s\rightarrow 0} \left(\ds{\frac{X_A^i( \widehat{\sigma}[\pi_A^s](t^1,\dots, \hat{t^A}, \dots, t^k)(r_A),u_A)l_As}{s}}
\right.\\
&&-\left.\ds{\frac{X_A^i( \widehat{\sigma}(t^1,\dots, \hat{t^A}, \dots, t^k)(r_A),u(t^1,\dots,r_A,\dots t^k))l_As +o(s)}{s}}\right)\\
&=&\lim_{s\rightarrow 0} \left(X_A^i( \widehat{\sigma}[\pi_A^s](t^1,\dots, \hat{t^A}, \dots, t^k)(r_A),u_A)l_A \right.
\\&&\left.- X_A^i( \widehat{\sigma}(t^1,\dots, \hat{t^A}, \dots, t^k)(r_A),u(t^1,\dots,r_A,\dots t^k))l_A\right)\\
&=& [\widehat{X}_A^i(\widehat{\sigma}(t^1,\dots, r_A,\dots, t^k),u_A)-\widehat{X}_A^i(\widehat{\sigma}(t^1,\dots, r_A,\dots, t^k),
u(t^1,\dots, r_A,\dots, t^k))]l_A:=\widehat{v}^i[\pi_A]
\end{eqnarray*}
for each $i\in \{0_1,\dots, 0_k,1,\dots n\}$.
\label{prop:pertVector}
\end{proof}

Note that the tangent vector in Proposition~\ref{prop:pertVector} is in $\left(\tau^{k;A}_{\widehat{Q}}
(T^1_k\widehat{Q})\right)_{\widehat{\sigma}(t_0^1,\dots,t_0^k)}={\rm T}_{\widehat{\sigma}(t_0^1,\dots,t_0^k)}\widehat{Q}$. The vector $\widehat{v}[\pi_A]$ is called
the \textbf{Ath-elementary perturbation vector associated to the perturbation data $\pi_A=\{r_A,l_A,u_A\}$}. It is also
called an \textbf{Ath-perturbation vector of class I}.


Following the same lines as in~\cite{2009BM} we can define the associated Ath-perturbation vector obtained from
$c$ different Ath-perturbation data $\pi_{A_1},\dots, \pi_{A_c}$ with different and/or same  perturbation time $r_{A_1}, \dots, r_{A_c}$.

At each copy of the tangent bundle in the $k$-tangent bundle $T^1_k\widehat{Q}$, we construct an \textbf{Ath-tangent
perturbation cone}
\begin{equation}
K^A_t=\overline{{\rm conv}\left( \bigcup_{a<\tau \leq t} \left(\Phi^{\widehat{X}_A^{\{u[\pi_A^s]\}}}_{(t,\tau)}\right)_*{\mathcal V}^A_\tau
\right)}
 \label{eq:ConeAth}
\end{equation}
where ${\mathcal V}^A_\tau$ denotes the set of Ath-elementary perturbation vectors at $\tau$, $\left(\Phi^{\widehat{X}_A^{\{u[\pi_A^s]\}}}_{(t,\tau)}\right)_*$ is the pushforward of the flow of $\widehat{X}_A^{\{u[\pi_A^s]\}}$ with $\widehat{\sigma}(\tau)$ as initial condition at time $\tau$, $\overline{{\rm conv} W}$ denotes the closure of the convex hull of the set $W$.

\remark If the controls are only measurable and bounded, as usually assumed in control theory, all the perturbations and 
geometric elements such as vectors, cones, etc. that appear in the paper are only defined at Lebesgue times where the equality~\eqref{eq:LebesgueTime} is satisfied.

The definition of $k$ different perturbation cones in~\eqref{eq:ConeAth} implies that the perturbation data associated with different $A$th copies are not mixed. As proved in~\cite[Proposition 3.12]{2009BM}, the following result is true for the cones $K^A_t$ for every $A=1,\dots, k$.

\begin{prop}  Let $t^A\in [t_0^A,t_f^A]$. If  $v$ is a nonzero vector in the interior of $K_t^A$, then there
exists $\epsilon >0$ such that for every $s\in (0,\epsilon)$ there are $s'>0$ and a perturbation of the control
$u[\pi_A]$ such that \begin{equation*} \widehat{\sigma}[\pi_A^s](t^1,\dots, \hat{t^A}, \dots, t^k)(t^A)=\widehat{\sigma}(t^1,\dots, \hat{t^A}, \dots, t^k)(t^A)+ s'v.\end{equation*}
\label{Prop:Lemma2A}
\end{prop}

This proposition is essential to prove Pontryagin's Maximum Principle in $k$-symplectic formalism.

\begin{thm}[$k$-symplectic Pontryagin's Maximum Principle] If $\widehat{\phi}^*=(\widehat{\sigma}^*,u^*)\colon I_1\times \dots \times I_k\rightarrow
\widehat{Q}\times U$ is a solution of the extended k-symplectic optimal control problem $(\widehat{Q},U,\widehat{\mathbf{X}},F,\mathbf{I})$, Statement~\ref{State:HatkOCP}, such that $F$ satisfies assumptions 1 and 2, then there exists $(\widehat{\mathbf{\beta}},u)
\colon I_1\times \dots \times I_k\rightarrow (T^1_k)^*\widehat{Q}\times U$ along $\widehat{\sigma}^*$ such that
\begin{enumerate}
 \item $(\pi^A\circ \widehat{\beta},u)$ along $\widehat{\sigma}^*$ is a solution of~\eqref{eq:iXAHA} for each $A=1,\dots, k$;
\item the Hamiltonian $H_A\colon (T^1_k)^*\widehat{Q}\times U \rightarrow \mathbb{R}$ in~\eqref{eq:HA} along the optimal integral section
 is equal to the supremum of $H_A$ over the controls almost everywhere;
\item the supremum of the Hamiltonian $H_A\colon (T^1_k)^*\widehat{Q}\times U \rightarrow \mathbb{R}$ in~\eqref{eq:HA} along the 
optimal integral section is constant almost everywhere;
\item $\widehat{\beta}^A(\mathbf{t}) \neq 0 \in T^*_{\widehat{\sigma}^*(\mathbf{t})}\widehat{Q}$ for each $\mathbf{t}\in I_1\times \dots \times I_k$
and for every $A=1,\dots,k$;
\item $\beta^A_{0_1}(\mathbf{t}),\dots, \beta^A_{0_k}(\mathbf{t})$ are constant and $\beta^A_{0_A}$ is non-positive for every $A=1,\dots,k$ .
\end{enumerate}

\begin{proof} As $(\widehat{\sigma}^*,u^*)$ is a solution of the extended $k$-symplectic optimal control problem,  if $\mathbf{\tau}\in I_1\times \dots \times I_k$, for every initial condition $\widehat{\beta}_{\mathbf{\tau}}$ in $(T^1_k)^*\widehat{Q}$ there exists a unique curve $\widehat{\beta}$ in $(T^1_k)^*\widehat{Q}$ satisfying the $k$ equations in~\eqref{eq:iXAHA} and the initial condition.

As in the classical Pontryagin's Maximum Principle the initial condition must be conveniently chosen so that the rest of conditions in the theorem are fulfilled.

For each $A\in \{1,\dots,k\}$, consider the Ath-tangent perturbation cone $K^A_{t_f^A}\subseteq T_{\widehat{\sigma}^*(\mathbf{t}_f)}\widehat{Q}$ and the vector $(0,\dots, \stackrel{A}{-1},\dots, 0,0,\stackrel{n}{\dots},0)$ in $ T_{\widehat{\sigma}^*(\mathbf{t}_f)}\widehat{Q}$ that indicates the decreasing direction of the coordinate $q^{0_A}(\mathbf{t})=\int_{t_0^A}^{t^A} F(\widehat{\sigma}^*(t^1,\dots,h,\dots, t^k),u^*(t^1,\dots,h,\dots, t^k)){\rm d}h$.

Observe that if  $(0,\dots, \stackrel{A}{-1},\dots, 0,0,\stackrel{n}{\dots},0)$ was in the interior of $K^A_{t_f^A}$, then 
there would exist an Ath-perturbation data $\pi_A=\{r_A,l_A,u_A\}$ such that $(\widehat{\sigma}[\pi_A],u[\pi_A])$ passes through the same points on $Q$ as $\sigma^*=\pi_2\circ \widehat{\sigma}^*$, but $q^{0_A}[ \widehat{\sigma}[\pi_A]](\mathbf{t}_f)< q^{0_A}[ \widehat{\sigma}^*](\mathbf{t}_f)$. This is a contradiction with the fact that  $(\widehat{\sigma}^*,u^*)$ is a solution of the extended k-symplectic optimal control problem. Hence  $(0,\dots, \stackrel{A}{-1},\dots, 0,0,\stackrel{n}{\dots},0)$ cannot be in the interior of $K^A_{t_f^A}$.

Thus, there exists $\widehat{\beta}^A_{t_f^A}\in T^*_{\widehat{\sigma}^*(\mathbf{t}_f)}\widehat{Q}$  such that
\begin{eqnarray}
\langle \widehat{\beta}^A_{t_f^A}, (0,\dots, \stackrel{A}{-1},\dots, 0,0,\stackrel{n}{\dots},0) \rangle &\geq &0, \label{Eq:Sep1}\\
\langle \widehat{\beta}^A_{t_f^A}, \widehat{v}[\pi_A] \rangle &\leq &0 \quad  \forall \; \widehat{v}[\pi_A]\in K^A_{t_f^A}. \label{Eq:Sep2}
\end{eqnarray}
Condition~\eqref{Eq:Sep1} implies that $\beta^A_{0_A}\leq 0$. Let us explicitly write for each $A\in \{1,\dots,k\}$ the equations for the integral curves of $\widehat{X}_A^*$ that satisfy equation~\eqref{eq:iXAHA}:
\begin{equation*}
\begin{array}{lcllcl}
\ds{\frac{\partial q^{0_B}}{\partial t^A}}&=&F\delta_{A}^B, \quad &\ds{ \frac{\partial p^A_{0_B}}{\partial t^A}}&=&0,\\
\ds{\frac{\partial q^{i}}{\partial t^A}}&=&X^i_{A}, \quad &\ds{ \frac{ \partial p^A_i}{\partial t^A}}&=&\ds{-\frac{\partial H_A}{\partial q^i}=-p^A_{0_A}\frac{\partial F}{\partial q^i}-p^A_j\frac{\partial X^j_A}{\partial q^i}},
\end{array}
\end{equation*}
for $B=1,\dots,k$ and $i,j=1,\dots,n$. Note that there are no equations for $p^B_i$ with $B\neq A$. Hence the momenta whose 
coordinates are $p^B_i$ remain undetermined. They will be determined by solving~\eqref{eq:iXAHA} with $A=B$.  Given an initial 
condition in the Ath-copy of $T^*_{\widehat{\sigma}^*(t_f)} \widehat{Q}$, we just solve the equations in the fiber for $p^A$.

If $\widehat{\beta}^A_{t_f^A}=\mathbf{0}\in  T^*_{\widehat{\sigma}^*(\mathbf{t}_f)}\widehat{Q}$, then the solution 
to~\eqref{eq:iXAHA} in the fiber will be zero along $\widehat{\sigma}^*$ because of the linearity of the differential equation
in the momenta. If the momenta is zero, it does not provide us with any information related to the separation condition. Hence, 
$\widehat{\beta}^A(\mathbf{t})\neq \mathbf{0}\in  T^*_{\widehat{\sigma}^*(\mathbf{t})}\widehat{Q}$ for every $\mathbf{t}\in \mathbf{I}$.

As in the classical Pontryagin's Maximum Principle, condition~\eqref{Eq:Sep2} and the definition of the Ath-elementary perturbation vector in~\eqref{Eq:PertVectorA} prove the condition about the supremum of the Hamiltonian $H_A\colon (T^1_k)^*\widehat{Q}\times U \rightarrow \mathbb{R}$ in~\eqref{eq:HA} over the controls for each $A\in \{1,\dots,k\}$. The constancy of the supremum of the Hamiltonian over the controls is proved analytically, analogously to the classical Pontryagin's Maximum Principle, see~\cite{2009BM,P62,S98Free} for more details.

From equations~\eqref{eq:iXAHA} we deduce that for each $A\in \{1,\dots,k\}$, $\beta^A_{0_B}$ is constant for every 
$B\in \{1,\dots,k\}$ along the optimal integral section $\widehat{\phi}^*=(\widehat{\sigma}^*,u^*)$.
\end{proof}
\label{thm:kPMP}
\end{thm}

\section{Application of unified formalism for $k$-cosymplectic to  implicit PDEs} \label{Sec:UnifiedKSymplectic}

We are going to apply the unified Skinner-Rusk formalism for $k$-cosymplectic field theories developed in~\cite[Section 4]{2012BcnSantiago} to the dynamics description for systems given by implicit partial differential equations. 
This will be very useful to develop Sections~\ref{Sec:UnifControlImplicitEDPS} and~\ref{Sec:example} so that physical examples
associated with higher order control partial differential equations fit in the approach considered in this paper.

The Whitney sum $T^1_kQ \oplus \left(T^1_k\right)^*Q$ has natural bundle structures over $T^1_kQ$ and $\left(T^1_k\right)^*Q$. The suitable bundle to describe non-autonomous dynamical systems governed by partial differential equations is ${\mathcal W}\colon = \mathbb{R}^k\times \left( T^1_kQ \oplus \left(T^1_k\right)^*Q\right)$. Local coordinates for ${\mathcal W}$ are $(t^B,q^i,v^i_A,p^A_i)$.  Let us denote by ${\rm pr}_1\colon \mathbb{R}^k\times \left(T^1_kQ \oplus \left(T^1_k\right)^*Q  \right)\rightarrow \mathbb{R}^k$, ${\rm pr}_2\colon \mathbb{R}^k\times \left(T^1_kQ \oplus \left(T^1_k\right)^*Q \right) \rightarrow T^1_kQ $ and
  ${\rm pr}_3\colon \mathbb{R}^k\times \left( T^1_kQ \oplus \left(T^1_k\right)^*Q \right) \rightarrow \left(T^1_k\right)^*Q$ 
the local projections into the first, second and third factor of ${\mathcal W}$, respectively. Locally, 
\begin{eqnarray*}
{\rm pr}_1(t^B,q^i,v^i_A,p^A_i)&=&t^B,\\
{\rm pr}_2(t^B,q^i,v^i_A,p^A_i)&=&(q^i,v^i_A)=(q,\mathbf{v}), \\
{\rm pr}_3(t^B,q^i,v^i_A,p^A_i)&=&(q^i,p^A_i)=(q,\mathbf{p}).
\end{eqnarray*}

Let $({\rm d}t^1,\dots,{\rm d}t^k)$ and $(\omega_1,\dots, \omega_k)$ be the canonical forms on $\mathbb{R}^k\times (T^1_k)^*Q$. We denote by $(\vartheta^1,\dots,\vartheta^l)$ and $(\Omega_1,\dots,\Omega_k)$ the pullback by ${\rm pr}_1$ and ${\rm pr}_3$ of these forms to $\mathbb{R}^k\times \left(T^1_kQ \oplus \left(T^1_k\right)^*Q \right)$, that is, $\vartheta^A=({\rm pr}_1)^*({\rm d}t^A)$ and  $\Omega_A=({\rm pr}_3)^*(\omega_A)$ for $1\leq A \leq k$. Locally,
\begin{equation} \vartheta^A={\rm d}t^A, \quad \Omega_A={\rm d}q^i \wedge {\rm d}p^A_i.
\label{eq:Omega0ALocal}
\end{equation}
The coupling function ${\mathcal C}$ on $\mathbb{R}^k\times \left( T^1_kQ \oplus \left(T^1_k\right)^*Q \right)$ is defined as follows
\begin{eqnarray*}
{\mathcal C}\colon \qquad \mathbb{R}^k \times T^1_kQ \oplus \left(T^1_k\right)^*Q  &  \longrightarrow & \mathbb{R}\\
(\mathbf{t},\mathbf{v}_{q},\mathbf{p}_{q})  & \longmapsto & \sum_{A=1}^k p^A_{q}(v_{A_q})=\sum_{A=1}^k \left(p^A_iv^i_A\right).
\end{eqnarray*}
Given a Lagrangian function $\mathbb{L}$ on  $\mathbb{R}^k\times T^1_kQ$, the Hamiltonian function $\mathbf{H}$ on  
$\mathbb{R}^k\times \left(T^1_kQ \oplus \left(T^1_k\right)^*Q \right)$ is defined as follows
\begin{equation}\label{eq:HunifSR}
H={\mathcal C}-({\rm pr}_1 \times {\rm pr}_2)^*\mathbb{L}.
\end{equation}
Locally, $H(t,q^i,v^i_A,p^A_i)=p^A_iv^i_A-\mathbb{L}(t,q^i,v^i_A)$.

The problem in  the Skinner-Rusk formalism for $k$-cosymplectic field theories  consists of finding integral sections $\phi\colon \mathbb{R}^k\rightarrow \mathbb{R}^k\times \left(T^1_kQ \oplus \left(T^1_k\right)^*Q \right)$ of an integrable $k$-vector field $\mathbf{Z}=(Z_1,\dots, Z_k)$  on $ \mathbb{R}^k\times \left(T^1_kQ \oplus \left(T^1_k\right)^*Q \right)$ such that:
\begin{equation}\label{eq:DynUnifSR}
\sum_{A=1}^k {\rm i}_{Z_A} \Omega_A={\rm d}H-\sum_{A=1}^k\dfrac{\partial H}{\partial t^A}\vartheta^A, \quad {\rm i}_{Z_A} \vartheta^B=\delta^B_A.
\end{equation}
See~\cite{2012BcnSantiago} for more details.



After summarizing briefly the Skinner-Rusk formalism for $k$-cosymplectic field theories, here we are interested in adapting it to find the dynamics of systems described by implicit partial differential equations. An implicit dynamical system  $(\mathbb{L},M)$  is  described by the submanifold
\begin{equation*}
M=\{(t^B,q^i,v^i_A)\in \mathbb{R}^k\times (T^1_k) Q \; | \; \Psi^\alpha(t^B,q^i,v^i_A)=0, \; 1\leq \alpha \leq s \},
\end{equation*}
of $\mathbb{R}^k\times T^1_k Q$, where ${\rm d}\Psi^1\wedge \dots \wedge {\rm d} \Psi^s\neq 0$, and a Lagrangian function $\mathbb{L}\in {\mathcal C}^\infty(M)$. This submanifold $M$ of $\mathbb{R}^k\times T^1_k Q$ can be naturally embedded by $\iota^M\colon M \hookrightarrow \mathbb{R}^k\times T^1_k Q$. 

In order to adapt the above formalism to this kind of dynamical systems, we must define the $k$-symplectic implicit bundle ${\mathcal W}^M=M\times_Q (T^1_k)^*Q$ and the corresponding canonical immersion
\begin{equation}
{\rm i}^{M}\colon {\mathcal W}^M \hookrightarrow {\mathcal W}\colon =\mathbb{R}^k\times \left( T^1_kQ \oplus \left(T^1_k\right)^*Q\right). \label{eq:DefineM}
\end{equation}
Now we can consider the pullback of the coupling function and the canonical forms on ${\mathcal W}$ to $ {\mathcal W}^M$:
\begin{equation*}
{\mathcal C}^{{\mathcal W}^M}=({\rm i}^M)^*({\mathcal C}), \quad \vartheta^A_{{\mathcal W}^M}=({\rm i}^M)^*(\vartheta^A),\quad  \Omega_A^{{\mathcal W}^M}=({\rm i}^M)^*(\Omega_A).
\end{equation*}

Let $\rho_1^M \colon {\mathcal W}^M \rightarrow M$ be the natural projection,  we define the Hamiltonian function $H_{{\mathcal W}^M}\colon {\mathcal W}^M \rightarrow \mathbb{R}$
as follows
\begin{equation*}
H_{{\mathcal W}^M}={\mathcal C}^{{\mathcal W}^M}-(\rho_1^M)^*\mathbb{L}.
\end{equation*}

Analogously to~\eqref{eq:DynUnifSR}, the problem of describing the dynamics of $(\mathbb{L},M)$ consists of finding  the integral sections $\phi\colon \mathbb{R}^k \rightarrow {\mathcal W}^M$ of an integrable 
$k$-vector field  $\mathbf{Z}=(Z_1,\dots,Z_k)$ on ${\mathcal W}^M$ such that
\begin{equation}
\sum_{A=1}^k {\rm i}_{Z_A}\Omega_A^{{\mathcal W}^M}={\rm d}H_{{\mathcal W}^M}-\sum_{A=1}^k\dfrac{\partial H_{{\mathcal W}^M}}{\partial t^A}\vartheta^A_{{\mathcal W}^M}, 
\quad {\rm i}_{Z_A}{\rm d}t^B=\delta^B_A. \label{eq:kSymplecticImplicit-1onWM}
\end{equation}
Or equivalently, the problems consists of finding  the integral sections $\phi\colon \mathbb{R}^k \rightarrow {\mathcal W}$ of an integrable 
$k$-vector field  $\mathbf{Z}=(Z_1,\dots,Z_k)$ on ${\mathcal W}=\mathbb{R}^k\times \left(T^1_kQ \oplus \left(T^1_k\right)^*Q\right)$ such that
\begin{equation}
\sum_{A=1}^k {\rm i}_{Z_A}\Omega_A={\rm d}H_{{\mathcal W}^M}-\sum_{A=1}^k\dfrac{\partial H_{{\mathcal W}^M}}{\partial t^A}\vartheta^A+\lambda_\alpha{\rm d} \Psi^\alpha-\lambda_\alpha\sum_{A=1}^k\dfrac{\partial \Psi^\alpha}{\partial t^A}\vartheta^A , 
\quad {\rm i}_{Z_A}\vartheta^B=\delta^B_A. \label{eq:kSymplecticImplicit-1}
\end{equation}
This equation is obtained from~\eqref{eq:kSymplecticImplicit-1onWM} by rewritting the equations on ${\mathcal W}$ so that the constraints $\psi^{\alpha}=0$ in~\eqref{eq:DefineM} must be added to the equation in a suitable way.

 The Ath vector field $Z_A$ on ${\mathcal W}$ is locally given by
\begin{equation*}
Z_A=(Z_A)^B_t \dfrac{\partial }{\partial t^B}+(Z_A)^i \dfrac{\partial }{\partial q^i}+(Z_A)^i_B \dfrac{\partial }{\partial v^i_B}+(Y_A)^B_i \dfrac{\partial}{\partial p^B_i}.
\end{equation*}
From~\eqref{eq:kSymplecticImplicit-1} we first have $(Z_A)^A_t=1$, $(Z_A)^B_t=0$ for $B\neq A$. Moreover, 
\begin{eqnarray*}
\sum_{A=1}^k {\rm i}_{Z_A}\Omega_A&-&{\rm d}H_{{\mathcal W}^M}+\sum_{A=1}^k\dfrac{\partial H_{{\mathcal W}^M}}{\partial t^A}\vartheta^A-\lambda_\alpha{\rm d} \Psi^\alpha+\lambda_\alpha\sum_{A=1}^k\dfrac{\partial \Psi^\alpha}{\partial t^A}\vartheta^A \\
&=& (Z_A)^i {\rm d}p^A_i-\left(\sum_{A=1}^k(Y_A)^A_i\right) {\rm d}q^i-v^i_A {\rm d}p_i^A-p^A_i {\rm d}v^i_A +\dfrac{\partial \mathbb{L}}{\partial q^i} {\rm d}q^i+\dfrac{\partial \mathbb{L}}{\partial v^i_A} {\rm d}v^i_A\\
&-&\lambda_\alpha \dfrac{\partial \Psi^\alpha}{\partial q^i}{\rm d}q^i-\lambda_\alpha \dfrac{\partial \psi^\alpha}{\partial v^i_A}{\rm d}v^i_A=0.
\end{eqnarray*}
Thus, 
\begin{eqnarray}
 (Z_A)^i&=&  v^i_A, \label{eq:ImplicitViA} \\
 \sum_{A=1}^k(Y_A)^A_i&=& \dfrac{\partial \mathbb{L}}{\partial q^i}-\lambda_\alpha \dfrac{\partial \Psi^\alpha}{\partial q^i} , \label{Eq:YAAiKSymplImplicit-1} \\
p^A_i&=& \dfrac{\partial \mathbb{L}}{\partial v^i_A}-\lambda_\alpha \dfrac{\partial \Psi^\alpha}{\partial v^i_A}. \label{Eq:pAiKSymplImplicit-1}
\end{eqnarray}
By also imposing the conditions~\eqref{eq:0Condition1} and~\eqref{eq:0Condition2} in the $k$-vector field on ${\mathcal W}$, we have
\begin{equation*}
Z_A= \dfrac{\partial }{\partial t^A}+v_A^i \dfrac{\partial }{\partial q^i}+ (Z_A)^i_B \dfrac{\partial }{\partial v^i_B}+ (Y_A)^A_i \dfrac{\partial }{\partial p^A_i},
\end{equation*}
with
\begin{equation*}
p^A_i=\dfrac{\partial \mathbb{L}}{\partial v^i_A}-\lambda_\alpha \dfrac{\partial \Psi^\alpha}{\partial v^i_A}, \quad 
 \sum_{A=1}^k(Y_A)^A_i= \dfrac{\partial \mathbb{L}}{\partial q^i}-\lambda_\alpha \dfrac{\partial \Psi^\alpha}{\partial q^i}.
\end{equation*}

If $\mathbf{Z}$ is a solution of~\eqref{eq:kSymplecticImplicit-1}, then we must start a constraint algorithm in the sense of~ \cite{1978Gotay}. To be more precise,  each $Z_A$ must be tangent to the submanifold $M_L$  contained in ${\mathcal W}^M$ and defined by~\eqref{Eq:pAiKSymplImplicit-1}. That is, the following tangency conditions must be satisfied
\begin{eqnarray}
0&=&Z_A(\Psi^\alpha), \label{Eq:TangCond1}\\
0&=&Z_A\left( p^B_i-\dfrac{\partial \mathbb{L}}{\partial v^i_B}+\lambda_\alpha \dfrac{\partial \Psi^\alpha}{\partial v^i_B}\right), \label{Eq:TangCond2}
 \end{eqnarray}
 on $M_L$ for every $A=1,\dots,k$. Depending on the particular examples, some components will be determined and the constraint algorithm must proceed until stabilization. 

%
\remark For non-autonomous $k$-symplectic explicit dynamical systems the manifold $M$ is defined by constraints locally given by $\Psi^\alpha(t,q^i,v^i_A)=v^i_A-X^i_A(t,q^i)=0$. Hence the above process can be used for this kind of dynamical systems.

\section{Unified formalism for optimal control problems governed by an implicit partial differential equation}\label{Sec:UnifControlImplicitEDPS}

We extend now the unified formalism for implicit control systems developed in \cite[Section 4]{2007SRNuestro} to 
optimal control problems whose dynamics is given by implicit control partial differential equations, instead of just explicit control partial differential equations as developed in Section~\ref{Sec:Setting}. The problem consists of finding the solutions to optimal control problems governed by an implicit partial differential equation by taking advantage of the unified formalism developed in this section. See Section~\ref{Sec:example} for a particular problem where this unified formalism is used.

Let $C$ be the control bundle with natural coordinates $(t^A,q^i,u^a)$. In contrast to the explicit description of control partial differential equations in Section~\ref{Sec:Setting}, let us consider now the case where the control partial differential equations are given implicitly by the following submanifold
\begin{equation*}
M_C=\{(t^B,u^a,q^i,v^i_A)\in  C\times_Q T^1_k Q  \; | \; \Psi^\alpha(t^B,u^a,q^i,v^i_A)=0, \; 1\leq \alpha \leq s \}
\end{equation*}
of $ C\times_Q T^1_k Q$, where ${\rm d}\Psi^1\wedge \dots \wedge {\rm d} \Psi^s\neq 0$. There exists a natural embedding  $\iota^{M_C}\colon M_C \hookrightarrow C\times_Q  T^1_k Q $. Then the implicit optimal control problem under consideration is determined by $(\mathbb{L},M_C)$, where $\mathbb{L}\in {\mathcal C}^\infty(M_C)$ is a Lagrangian function. 

Let us define now the $k$-symplectic implicit control bundle ${\mathcal W}^{M_C}=M_C\times_Q (T^1_k)^*Q$ which is a submanifold of $C\times_{\mathbb{R}^k \times Q} {\mathcal W}=C\times_{\mathbb{R}^k \times Q} \left(\mathbb{R}^k\times \left( T^1_k Q \oplus (T^1_k)^*Q\right)\right)$. Then we have, respectively, the canonical immersion and the natural projection:
\begin{equation*}
{\rm i}^{M_C}\colon {\mathcal W}^{M_C} \hookrightarrow C\times_{\mathbb{R}^k \times Q} {\mathcal W},  \quad 
\sigma_{\mathcal W}\colon C\times_{\mathbb{R}^k \times Q} {\mathcal W} \rightarrow {\mathcal W}.
\end{equation*}
Now we can consider the pullback of the coupling function in Section~\ref{Sec:UnifiedKSymplectic} and the canonical forms on $ {\mathcal W}$ to $ {\mathcal W}^{M_C}$:
\begin{equation*}
{\mathcal C}^{{\mathcal W}^{M_C}}=(\sigma_{\mathcal W}\circ {\rm i}^{M_C})^*({\mathcal C}), \quad \Omega_A^{{\mathcal W}^{M_C}}=(\sigma_{\mathcal W}\circ {\rm i}^{M_C})^*(\Omega_A), \quad \vartheta^A_{{\mathcal W}^{M_C}}=(\sigma_{\mathcal W}\circ {\rm i}^{M_C})^*(\vartheta^A)
\end{equation*}

Let $\rho_1^{M_C} \colon {\mathcal W}^{M_C} \rightarrow M_C$ be the natural projection, the Hamiltonian function $H_{{\mathcal W}^{M_C}}\colon {\mathcal W}^{M_C} \rightarrow \mathbb{R}$ is defined as follows
\begin{equation*}
H_{{\mathcal W}^{M_C}}={\mathcal C}^{{\mathcal W}^{M_C}}-(\rho_1^{M_C})^*\mathbb{L}.
\end{equation*}
%
The dynamics of the optimal control problem $(\mathbb{L},M_C)$ is determined by the solutions of the equations
\begin{equation}
\sum_{A=1}^k{\rm i}_{Z_A}\left(\Omega_A^{{\mathcal W}^{M_C}}\right)=0, \quad {\rm i}_{Z_A}\vartheta^B_{{\mathcal W}^{M_C}}=\delta^B_A, \label{eq:ImplkOCP}
\end{equation}
for a $k$-vector field $\mathbf{Z}=(Z_1,\dots, Z_k)$ on ${\mathcal W}^{M_C}$.

In order to work in local coordinates we need the following proposition whose proof is straightforward.

\begin{prop}
For a given $w\in {\mathcal W}^{M_C}$, the
following conditions are equivalent:
\begin{enumerate}
\item[(1)]  There exists  a $k$-vector field  $\mathbf{Z}_w \in (T^1_k)_w{\mathcal W}^{M_C}$ verifying that
\[
\sum_{A=1}^k\Omega_A^{{\mathcal W}^{M_C}}((Z_A)_w, (Y_A)_w) =0\ , \ \mbox{\rm for every
$\mathbf{Y}_w\in (T^1_k)_w {\mathcal W}^{M_C}$} \ .
\]
\item[(2)]  There exists a $k$-vector field $\mathbf{Z}_w\in (T^1_k)_w (C\times_{\mathbb{R}^k\times Q} {\mathcal W})$
verifying that
 \begin{enumerate}
\item[(i)] $\mathbf{Z}_w\in (T^1_k)_w {\mathcal W}^{M_C}$,
\item[(ii)] $\sum_{A=1}^k{\rm i}_{(Z_A)_w}(\sigma_{\cal W}^*(\Omega_A))_w\in ((T^1_k)_w {\mathcal W}^{M_C})^0$
\ .
\end{enumerate}
\end{enumerate}
\label{2cons}
\end{prop}

As a consequence of this last proposition, we can obtain
the implicit optimal control equations using condition (2) in Proposition~\ref{2cons} as
follows: there exists a $k$-vector field $\mathbf{Z}$ on $C\times_{\mathbb{R}^k\times Q}{\mathcal W}$ such
that \begin{itemize}
\item[(i)]
$\mathbf{Z}$ is tangent to ${\mathcal W}^{M_C}$;
\item[(ii)]
the $1$-form $\sum_{A=1}^k{\rm i}_{(Z_A)}(\sigma_{\cal W}^*(\Omega_A))$
is null on the $k$-vector fields tangent to ${\mathcal W}^{M_C}$.
\end{itemize}
As ${\mathcal W}^{M_C}=M_C\times_Q ({\rm T}^1_k)^* Q$  and the constraints are
 $\Psi^\alpha=0$; then there exist
$\lambda_{\alpha} \in{\mathcal C}^\infty(C\times_{\mathbb{R}^k\times Q}{\mathcal W})$,
to be determined, such that
\begin{equation}
\sum_{A=1}^k{\rm i}_{(Z_A)}(\sigma_{\cal W}^*(\Omega_A))\vert_{{\mathcal
W}^{M_C}}= \left({\rm d} H_{{\mathcal W}^{M_C}}-\dfrac{\partial H_{{\mathcal W}^{M_C}}}{\partial t^A} \vartheta^A + \lambda_{\alpha}{\rm d}\Psi^{\alpha} -\lambda_\alpha \dfrac{\partial \Psi^\alpha}{\partial t^A} \vartheta^A\right)\vert_{{\mathcal W}^{M_C}}.  \label{Eq:iZAW0MC}
\end{equation}
As usual, the undetermined functions $\lambda_{\alpha}$'s
are called Lagrange multipliers.

Now using coordinates $(t^B, u^a,q^i,v^i_A,p^A_i)$ in $C\times_{\mathbb{R}^k\times Q}
{\mathcal W}$, we look for $k$ vector fields
\[
Z_A=(Z_A)^B_t\dfrac{\partial}{\partial t^B}+(Z_A)^a\dfrac{\partial}{\partial u^a}+(Z_A)^i\dfrac{\partial}{\partial q^i}+(Z_A)^i_B\dfrac{\partial}{\partial v^i_B}
+(Y_A)^B_i\frac{\partial}{\partial p^B_i} \ ,
\]
where $(Z_A)^B_t$, $(Z_A)^a$, $(Z_A)^i$, $(Z_A)^i_B$, $(Y_A)^B_i$ are unknown functions on
${\mathcal W}^{M_C}$ verifying the equation
\begin{eqnarray*}
0&=&\sum_{A=1}^k{\rm i}_{Z_A}\left({\rm d} q^i\wedge {\rm d} p_i^A\right)- {\rm d} \left(\sum_{A=1}^k (p_i^Av^i_A)-\mathbb{L} (t,u,q,\mathbf{v})\right) +\dfrac{\partial H_{{\mathcal W}^{M_C}}}{\partial t^A} {\rm d}t^A \\ 
&&-\lambda_{\alpha}{\rm d}\Psi^{\alpha} +\lambda_\alpha \dfrac{\partial \Psi^\alpha}{\partial t^A} {\rm d}t^A\\
&=& \left( -\sum_{A=1}^k(Y_A)^A_i-\lambda_\alpha \dfrac{\partial \Psi^\alpha}{\partial q^i}+ \dfrac{\partial \mathbb{L}}{\partial q^i}\right){\rm d}q^i + \left(\frac{\partial \mathbb{L} }{\partial u^a}
-\lambda_{\alpha}\frac{\partial\Psi^{\alpha}}{\partial u^a}\right){\rm d} u^a
\\ &&+\left( - p^A_i+ \dfrac{\partial \mathbb{L}}{\partial v^i_A}-
\lambda_{\alpha}\dfrac{\partial \Psi^{\alpha}}{\partial v^i_A}\right){\rm d} v^i_A+((Z_A)^i-v^i_A) {\rm d} p_i^A .
\end{eqnarray*}
Note that from~\eqref{eq:ImplkOCP} we have $(Z_A)^A_t=1$ and $(Z_A)^B_t=0$ for $A\neq B$. Moreover,
\begin{eqnarray}
 \sum_{A=1}^k(Y_A)^A_i&=& \dfrac{\partial \mathbb{L}}{\partial q^i}-\lambda_\alpha \dfrac{\partial \Psi^\alpha}{\partial q^i}, \label{eq:SROptControl1}\\ [2mm]
 (Z_A)^i&=&  v^i_A, \label{eq:SROptControl2} \\  [2mm]
p^A_i&=& \dfrac{\partial \mathbb{L}}{\partial v^i_A} -\lambda_\alpha \dfrac{\partial \Psi^\alpha}{\partial v^i_A}, , \label{eq:SROptControl3}\\[2mm]
 0&=& \dfrac{\partial \mathbb{L}}{\partial u^a}-\lambda_\alpha \dfrac{\partial \Psi^\alpha}{\partial u^a}. , \label{eq:SROptControl4}
\end{eqnarray}
together with the tangency conditions
 \begin{equation}
0=Z_A(\Psi^{\alpha})\vert_{{\mathcal W}^{M_C}}=
\left((Z_A)^A_t\frac{\partial \Psi^{\alpha}}{\partial t^A}+
(Z_A)^i\frac{\partial \Psi^{\alpha} }{\partial q^i}+
(Z_A)^a\frac{\partial \Psi^{\alpha} }{\partial u^a}+
(Z_A)^i_B\frac{\partial \Psi^{\alpha}}{\partial v^i_B}\right)\Big\vert_{{\mathcal W}^{M_C}} \label{eq:TangCondSRControl-1}
 \end{equation}
for every $A=1,\dots, k$. Imposing the conditions~\eqref{eq:0Condition1},~\eqref{eq:0Condition2}, we know that $(Y_A)^B_i=0$ for $A\neq B$. From here we can start a constraint algorithm in the sense of~\cite{1978Gotay} as follows: the  tangency conditions with respect to the constraints~\eqref{eq:SROptControl3} and~\eqref{eq:SROptControl4}  obtained from equation~\eqref{Eq:iZAW0MC} give the following equations on $\mathcal{W}^{M_C}$:
\begin{eqnarray*}
0&=&Z_A\left(p^B_i-\dfrac{\partial \mathbb{L}}{\partial v^i_B}+\lambda_\alpha \dfrac{\partial \psi^\alpha}{\partial v^i_B}\right)\\
&=& (Y_A)^A_i\delta_A^B-\dfrac{\partial^2 \mathbb{L}}{\partial t^A \partial v^i_B} -\dfrac{\partial^2 \mathbb{L}}{\partial q^j \partial v^i_B} v^j_A-  \dfrac{\partial^2 \mathbb{L}}{\partial v^j_C \partial v^i_B} (Z_A)^j_C- \dfrac{\partial^2 \mathbb{L}}{\partial u^a \partial v^i_B} (Z_A)^a\\&&+\lambda_\alpha \dfrac{\partial^2 \psi^\alpha}{\partial t^A \partial v^i_B} +\lambda_\alpha \dfrac{\partial^2 \psi^\alpha}{\partial q^j \partial v^i_B} v^j_A+ \lambda_\alpha \dfrac{\partial^2 \psi^\alpha}{\partial v^j_C \partial v^i_B} (Z_A)^j_C+ \lambda_\alpha \dfrac{\partial^2 \psi^\alpha}{\partial u^a \partial v^i_B} (Z_A)^a;\\ [3mm]
0&=&Z_A\left( \dfrac{\partial \mathbb{L}}{\partial u^a}-\lambda_\alpha \dfrac{\partial \Psi^\alpha}{\partial u^a}\right)\\
&=&  \dfrac{\partial^2 \mathbb{L}}{\partial t^A \partial u^a}+ \dfrac{\partial^2 \mathbb{L}}{\partial u^b \partial u^a}(Z_A)^b+  \dfrac{\partial^2 \mathbb{L}}{\partial q^i\partial u^a}v^i_A+ \dfrac{\partial^2 \mathbb{L}}{\partial v^i_B\partial u^a}(Z_A)^i_B-\lambda_\alpha \dfrac{\partial^2 \Psi^\alpha}{\partial t^A \partial u^a}-\lambda_\alpha \dfrac{\partial^2 \Psi^\alpha}{\partial q^i \partial u^a} v^i_A\\ && -\lambda_\alpha \dfrac{\partial^2 \Psi^\alpha}{\partial v^i_B \partial u^a} (Z_A)^i_B-\lambda_\alpha \dfrac{\partial^2 \Psi^\alpha}{\partial u^b \partial u^a} (Z_A)^b.
\end{eqnarray*}
If the square matrix of size $k+nk$
\begin{equation*}
\begin{pmatrix}
 \dfrac{\partial^2 \mathbb{L} }{\partial u^b \partial u^a}-\lambda_\alpha \dfrac{\partial^2 \Psi^\alpha}{\partial u^b \partial u^a} &  \dfrac{\partial^2 \mathbb{L} }{\partial v^i_B \partial u^a}-\lambda_\alpha \dfrac{\partial^2 \Psi^\alpha}{\partial v^i_B \partial u^a} \\[3mm] - \dfrac{\partial^2 \mathbb{L} }{\partial v^j_C \partial u^b}+ \lambda_\alpha \dfrac{\partial^2 \Psi^\alpha}{\partial u^b \partial v^j_C} & - \dfrac{\partial^2 \mathbb{L} }{\partial v^i_B \partial v^j_C}+\lambda_\alpha \dfrac{\partial^2 \Psi^\alpha}{\partial v^i_B \partial v^j_C}
\end{pmatrix}
\end{equation*}
has maximum rank, $(Z_A)^b$ and $(Z_A)^i_B$ are determined in terms of $(Y_A)^A_i$, which must satisfy the condition~\eqref{eq:SROptControl1}
 coming from~\eqref{Eq:iZAW0MC}. The algorithm continues until stabilization.

\section{Example: Orientation of a bipolar molecule in the plane by means of two external fields}\label{Sec:example}

Let us consider now the control partial differential equation studied in \cite[Section 8]{Marco}:
\begin{equation}
{\rm i} \dfrac{\partial \Psi(t,\theta)}{\partial t}=\left( -\dfrac{\partial^2 \Psi(t,\theta)}{\partial \theta^2}+u_1(t) \cos \theta \, \Psi(t,\theta)+u_2(t) \sin \theta \, \Psi(t,\theta)\right),
\label{eq:Control}
\end{equation}
where $\Psi$ is an element in a Hilbert space taking values on the complex and $u_1$, $u_2$ take values in $\mathbb{R}$.

Let us rewrite the problem according to Section~\ref{Sec:UnifControlImplicitEDPS}. This equation fits in 2-symplectic formalism where $t^1=t$ and $t^2=\theta$. Note that~\eqref{eq:Control} is a partial differential equation on the complex numbers. Hence let us rename
\begin{equation*}
q^1={\rm Re} \Psi, \quad q^2={\rm Im} \Psi.
\end{equation*}
In order to rewrite~\eqref{eq:Control} as an implicit partial differential equation we work on a 6-dimensional manifold $Q$ with local coordinates
\begin{equation*}
\left(q^1,q^2,q^3=\dfrac{\partial q^1}{\partial t}, q^4=\dfrac{\partial q^2}{\partial t},q^5=\dfrac{\partial q^1}{\partial \theta}, q^6=\dfrac{\partial q^2}{\partial \theta}\right)
\end{equation*}
to transform the second order partial differential equation into first order partial differential equations.

The local coordinates for $C\times_Q T^1_2Q$ are
 $(t^1,t^2,u^1,u^2,q^i,v_1^i,v_2^i)$. Note that apart from~\eqref{eq:Control} we also know that
\begin{equation}
v_1=\left( \dfrac{\partial q^1}{\partial t},\dfrac{\partial q^2}{\partial t},\dfrac{\partial^2 q^1}{\partial t \partial t},\dfrac{\partial^2 q^2}{\partial t \partial t},\dfrac{\partial^2 q^1}{\partial t \partial \theta },\dfrac{\partial^2 q^2}{\partial t \partial \theta}\right),\quad
v_2=\left( \dfrac{\partial q^1}{\partial \theta},\dfrac{\partial q^2}{\partial \theta},\dfrac{\partial^2 q^1}{\partial \theta \partial t},\dfrac{\partial^2 q^2}{\partial \theta \partial t},\dfrac{\partial^2 q^1}{\partial \theta \partial \theta },\dfrac{\partial^2 q^2}{\partial \theta \partial \theta}\right).\label{Eq:v1v2}
\end{equation}
Hence~\eqref{Eq:v1v2} determines some relationships between some coordinates of $v_1$ and $v_2$. Equations~\eqref{eq:Control} and~\eqref{Eq:v1v2}, determine a submanifold $M_C$ of $C\times_Q T^1_2Q$ implicitly defined by the following constraints:
\begin{equation*}
\begin{array}{rclrclrcl}
\Psi^1&=& v^1_1-q^3, \quad &\Psi^4&=&v^2_2-q^6, \quad & \Psi^7&=&-q^3-v^6_2+u_1 q^2 \cos \theta +u_2 q^2 \sin \theta\, , \\

\Psi^2&=&v^2_1-q^4, \quad & \Psi^5&=&v^3_2-v^5_1, \quad & \Psi^8&=&q^4-v^5_2+u_1 q^1 \cos \theta +u_2 q^1 \sin \theta \, . \\
\Psi^3&=&v^1_2-q^5, \quad & \Psi^6&=&v^4_2-v^6_1, \quad &&&
\end{array}
\end{equation*}
A general $2$-vector field $\mathbf{Z}$ on $C\times_{\mathbb{R}^2\times Q}\left(\mathbb{R}^2\times \left(T^1_2Q \oplus (T^1_2)^*Q\right)\right)$ is locally given by
\begin{equation*}
Z_A=(C_A)^1\dfrac{\partial}{\partial t}+(C_A)^2\dfrac{\partial}{\partial \theta}+(D_A)_a\dfrac{\partial}{\partial u_a}+(E_A)^i\dfrac{\partial}{\partial q^i}+(F_A)^i_B\dfrac{\partial}{\partial v^i_B}+(G_A)^B_i\dfrac{\partial}{\partial p^B_i}.
\end{equation*}
Assume that the cost function is control-quadratic in the following way $\mathbb{L}=\dfrac{1}{2}(u_1^2+u_2^2)$. From~\eqref{eq:ImplkOCP},~\eqref{Eq:iZAW0MC} we have

\begin{equation*}
\begin{array}{lcl}
{\rm i}_{Z_A}{\rm d}t^B=\delta_A^B & \longrightarrow &  (C_1)^1=1, \quad (C_1)^2=0, \quad (C_2)^1=0, \quad (C_2)^2=1,\\[3mm]
{\rm d}p^A_{i} & \longrightarrow & (E_A)^i= v^i_A,\\[3mm]
{\rm d}q^1 & \longrightarrow &
(G_1)^1_1+(G_2)^2_1=-\lambda_8(u_1\cos \theta +u_2 \sin \theta),\\[3mm]
 {\rm d}q^2 & \longrightarrow &
(G_1)^1_2+(G_2)^2_2=-\lambda_7(u_1\cos \theta +u_2 \sin \theta), \\[3mm]
 {\rm d}q^3 & \longrightarrow &
(G_1)^1_3+(G_2)^2_3=\lambda_1+\lambda_7, \\[3mm]
 {\rm d}q^4 & \longrightarrow &
(G_1)^1_4+(G_2)^2_4=\lambda_2-\lambda_8, \\[3mm]
 {\rm d}q^5 & \longrightarrow &
(G_1)^1_5+(G_2)^2_5=\lambda_3, \\[3mm]
 {\rm d}q^6 & \longrightarrow &
(G_1)^1_6+(G_2)^2_6=\lambda_4, \\[3mm]
{\rm d}v^i_A & \longrightarrow & \lambda_1+ p^1_1=0, \quad \lambda_2+ p^1_2=0, \quad \lambda_3+ p^2_1=0,
 \quad \lambda_4+\ p^2_2=0, \quad \lambda_5+ p^2_3=0,  \\[3mm] && -\lambda_5+ p^1_5=0, \quad  \lambda_6+ p^2_4=0,
 \quad -\lambda_6+ p^1_6=0, \quad -\lambda_7+  p^2_6=0, \quad -\lambda_8+ p^2_5=0,
\\[3mm] &&  p^1_3=0, \quad p^1_4=0,
 \\[3mm]
{\rm d}u_1 & \longrightarrow & \lambda_7 q^2 \cos \theta +\lambda_8 q^1 \cos \theta  - u_1=0, \\ [3mm]
{\rm d}u_2 & \longrightarrow & \lambda_7 q^2 \sin \theta +\lambda_8 q^1 \sin \theta  - u_2=0.
\end{array}
\end{equation*}
Hence all the controls and Lagrange multipliers are determined:
\begin{equation*}
\begin{array}{l}
\lambda_1=-p_1^1, \quad \lambda_2=-p_2^1, \quad \lambda_3=-p_1^2, \quad \lambda_4=-p_2^2, \quad
\lambda_5=-p_3^2=p^1_5, \quad \lambda_6=-p^2_4=p^1_6, \quad \lambda_7=p^2_6,   \\[3mm]
 \lambda_8=p^2_5,  \quad u_1=p^2_6q^2\cos \theta +p^2_5 q^1 \cos \theta, \quad u_2=p^2_6q^2\sin \theta +p^2_5 q^1 \sin \theta.
\end{array}
\end{equation*}
The cost function can be written as follows:
\begin{equation*}
\mathbb{L}=\dfrac{1}{2}(u_1^2+u_2^2)=(p^2_6q^2+p^2_5q^1)^2.
\end{equation*}
Since $p^1_3=0$ and $p^1_4=0$, we have $Z_A(p^1_3)=Z_A(p^1_4)=0$ for $A=1,2$. Then, $(G_1)^1_3=(G_2)^1_3=(G_2)^1_4=(G_1)^1_4=0$. Having this in mind, we have
\begin{equation}
\begin{array}{rcl}
 (G_1)^1_1&=&-(G_2)^2_1-p^2_5(p^2_6 q^2+p^2_5 q^1), \\ [3mm]
(G_1)^1_2&=&-(G_2)^2_2-p^2_6(p^2_6 q^2+ p^2_5 q^1), \\ [3mm]
(G_2)^2_3&=&-p^1_1+p^2_6, \\ [3mm] (G_2)^2_4&=&-p^1_2-p^2_5, \\ [3mm]
(G_1)^1_5&=&-(G_2)^2_5-p^2_1, \\ [3mm] (G_1)^1_6&=&-(G_2)^2_6-p^2_2.
\end{array}\label{eq:example1}
\end{equation}
Note that the controls satisfy the following relationship $u_1 \sin \theta= u_2 \cos \theta$. If we impose the tangency condition, we have
\begin{eqnarray*}
Z_1(u_1 \sin \theta-u_2 \cos \theta)&=& (D_1)_1 \sin \theta - (D_1)_2 \cos \theta=0,\\
Z_2(u_1 \sin \theta-u_2 \cos \theta)&=& u_1\cos \theta +u_2\sin \theta+(D_2)_1 \sin \theta - (D_2)_2 \cos \theta=0.
\end{eqnarray*}
Thus, $(D_1)_1=\cos \theta$, $(D_1)_2=\sin \theta$ and $(D_2)_1 \sin \theta - (D_2)_2 \cos \theta=-p^2_6q^2-p^2_5q^1$.

By conditions~\eqref{eq:0Condition1},~\eqref{eq:0Condition2} we have that $(G_A)^B_i=0$ for $A\neq B$.

After imposing the tangency conditions in~\eqref{eq:TangCondSRControl-1} we obtain
\begin{equation*}
\begin{array}{ll}
Z_A(\Psi^1)=(F_A)^1_1-v^3_A=0, &\quad Z_A(\Psi^2)=(F_A)^2_1-v^4_A=0,\\[3mm]
Z_A(\Psi^3)=(F_A)^1_2-v^5_A=0, &\quad Z_A(\Psi^4)=(F_A)^2_2-v^6_A=0,\\[3mm]
Z_A(\Psi^5)=(F_A)^3_2-(F_A)^5_1=0, &\quad Z_A(\Psi^6)=(F_A)^4_2-(F_A)^6_1=0,
\end{array}
\end{equation*}
\begin{equation*}
\begin{array}{l}
Z_A(\Psi^7)=-v^3_A-(F_A)^6_2+(D_A)_1 q^2\cos \theta -\delta^A_2 u_1 q^2\sin \theta+v^2_A u_1 \cos \theta +
(D_A)_2 q^2\sin \theta +\delta^A_2 u_2 q^2\cos \theta \\ [3mm]
\quad +v^2_A u_2 \sin \theta=0, \\[3mm]
Z_A(\Psi^8)=v^4_A-(F_A)^5_2+(D_A)_1 q^1\cos \theta -\delta^A_2 u_1 q^1\sin \theta+v^1_A u_1 \cos \theta +
(D_A)_2 q^1\sin \theta +\delta^A_2 u_2 q^1\cos \theta \\ [3mm]
\quad +v^1_A u_2 \sin \theta=0,
\end{array}
\end{equation*}
Thus,
\begin{small}
\begin{eqnarray*}
Z_A&=&\dfrac{\partial}{\partial t^A}+v_A^i\dfrac{\partial}{\partial q^i}+(D_A)_a\dfrac{\partial}{\partial u_a}+v^3_A \dfrac{\partial}{\partial v^1_1}+v^4_A \dfrac{\partial}{\partial v^2_1}+v^5_A \dfrac{\partial}{\partial v^1_2}+v^6_A \dfrac{\partial}{\partial v^2_2}\\&+&(F_A)^5_1 \left(\dfrac{\partial}{\partial v^3_2}+\dfrac{\partial}{\partial v^5_1}\right)+(F_A)^6_1 \left(\dfrac{\partial}{\partial v^4_2}+\dfrac{\partial}{\partial v^6_1}\right)\\&+&\left(v^4_A+(D_A)_1 q^1\cos \theta -\delta^A_2 u_1 q^1\sin \theta+v^1_A u_1 \cos \theta +
(D_A)_2 q^1\sin \theta +\delta^A_2 u_2 q^1\cos \theta
 +v^1_A u_2 \sin \theta\right)\dfrac{\partial}{\partial v^5_2} \\&+&\left( -v^3_A+(D_A)_1 q^2\cos \theta -\delta^A_2 u_1 q^2\sin \theta+v^2_A u_1 \cos \theta +
(D_A)_2 q^2\sin \theta +\delta^A_2 u_2 q^2\cos \theta
+v^2_A u_2 \sin \theta\right) \dfrac{\partial}{\partial v^6_2}\\&+&(F_A)^3_1\dfrac{\partial}{\partial v^3_1}+(F_A)^4_1\dfrac{\partial}{\partial v^4_1}+(G_A)^A_i\dfrac{\partial}{\partial p^A_i},
\end{eqnarray*}
\end{small}
where $t^1=t$, $t^2=\theta$, $(D_1)_1=\cos \theta$, $(D_1)_2=\sin \theta$ and $(D_2)_1 \sin \theta - (D_2)_2 \cos \theta=-p^2_6q^2-p^2_5q^1$ and also equations~\eqref{eq:example1} are satisfied. 
The optimal sections are integral sections of $\mathbf{Z}=(Z_1,\dots,Z_k)$.

\section{Future work}

After this first geometric approach to  optimal control problems governed by partial differential equations, it remains open to find the way to successfully extend any control system regardless of the nature of the cost function. The main difficulty is to obtain a compatible system of partial differential equations after extending the original control system. 

In this paper we have not mentioned the different kind of extremals for optimal control problems. There exist the so-called abnormal extremals which are characterized at first without considering the cost function. As shown in~\cite{2009BMAlgorithm}, the constraint algorithm in the sense of Gotay-Nester-Hinds is useful to characterize the different kind of extremals in optimal control theory. Now, that the optimal control problems  governed by partial differential equations have been understood in the $k$-symplectic framework, it seems that the application of the constraint algorithm for $k$-presymplectic Hamiltonian systems~\cite{2009XavierEtAl} will characterize the extremals of those problems.

\section*{Acknowledgements}

This work has been partially supported by MICINN (Spain)
Grants MTM2008-00689, MTM2009-08166, MTM2010-12116-E, MTM 2010-21186-C02-02; 2009SGR1338 of
the Catalan government, IRSES project GEOMECH (246981) within the 7th European Community Framework Program. MBL has been financially supported by Juan de la Cierva fellowship from MICINN.

\end{document}